\documentclass{article}
\usepackage{xcolor}
\usepackage{tabularx}
\usepackage{ragged2e}
\usepackage{longtable}
\usepackage{amsmath}
\usepackage{amsfonts}
\usepackage{amssymb}
\usepackage{amsthm}
\usepackage{listings}
\usepackage{graphicx, rotate}
\usepackage{seqsplit}

\author{Victor S. Miller}
\title{Counting Matrices That are Squares}
\date{April 2016}
\newcolumntype{R}[1]{>{\RaggedLeft\arraybackslash}p{#1}}
\newcolumntype{L}[1]{>{\RaggedRight\arraybackslash}p{#1}}
\definecolor{shadecolor}{rgb}{0.9,0.9,0.8}
\newcommand{\GF}[1][2]{\mathbb{F}_{#1}}
\newcommand{\ZZ}{\mathbb{Z}}
\newcommand{\CC}{\mathbb{C}}
\newcommand{\cA}{\mathcal{A}}
\newcommand{\cC}{\mathcal{C}}
\newcommand{\cP}{\mathcal{P}}
\newcommand{\cS}{\mathcal{S}}
\newcommand{\cI}{\mathcal{I}}
\newcommand{\cM}{\mathcal{M}}
\DeclareMathOperator{\GL}{GL}
\DeclareMathOperator{\Mat}{Mat}
\newtheorem{proposition}{Proposition}
\newtheorem{lemma}{Lemma}
\newtheorem{corollary}{Corollary}
\newtheorem{theorem}{Theorem}
\newtheorem{definition}{Definition}

\newcommand{\propref}[1]{Proposition~\ref{prop:#1}}
\newcommand{\lemref}[1]{Lemma~\ref{lem:#1}}
\newcommand{\thmref}[1]{Theorem~\ref{thm:#1}}
\newcommand{\corref}[1]{Corollary~\ref{cor:#1}}
\newcommand{\tabref}[1]{Table~\ref{tab:#1}}
\newcommand{\lstref}[1]{Listing~\ref{lst:#1}}
\newcommand{\figref}[1]{Figure~\ref{fig:#1}}
\newcommand{\secref}[1]{Section~\ref{sec:#1}}
\lstnewenvironment{pycode}[1][]{%
  \lstset{%
    language=Python,
    basicstyle=\tiny,
    breaklines=true,
    frame=single,
    framerule=0pt,
    mathescape=true,
    #1,
  }%
}{%
}
\lstnewenvironment{ccode}[1][]{%
  \lstset{
    basicstyle=\tiny,
    breaklines=true,
    frame=single,
    framerule=0pt,
    mathescape=true,
    #1,
  }%
}{%
}

\newenvironment{dedication}{\vspace{6ex}\begin{quotation}\begin{center}\begin{em}}%
{\par\end{em}\end{center}\end{quotation}}
\begin{document}
\maketitle
\begin{dedication}
Dedicated to Neil Sloane on his 75th Birthday
\end{dedication}
\begin{abstract}
On the math-fun mailing list (7 May 2013), Neil Sloane asked to calculate
  the number of $n \times n$ matrices with entries in $\{0,1\}$ which
  are squares of other such matrices.  In this paper we analyze the
  case 
  that the arithmetic is in $\mathbb{F}_{2}$. We follow
  the dictum of Wilf (``What is an answer?'') to derive a ``effective'' algorithm to count
  such matrices in much less time than it takes to enumerate them.
  The algorithm which we use
  involves the analysis of conjugacy classes of matrices.  The
  restricted integer partitions which arise are
  counted by the coefficients of one of Ramanujan's mock Theta
  functions, which we found thanks to Sloane's OEIS (Online
  Encyclopedia of Integer Sequences).

   Let $a_n$ be the number
  elements of ${\rm Mat}_n(\mathbb{F}_{2})$ which are squares, and $b_n$ be the
  number of elements of ${\rm GL}(n,\mathbb{F}_{2})$ which are squares.
  The numerical results strongly suggest that there are constants
  $\alpha,\beta > 0$ such that $a_n \sim \alpha 2^{n^2}$, $b_n \sim
  \beta 2^{n^2}$.
\end{abstract}

\section{Introduction}
\label{sec:intro}
On the math-fun mailing list (7 May 2013), Neil Sloane asked ``What is
$a(n) = $ the number of $n \times n$ matrices in $R$ that have a square
root in $R$'', where $R$ is the set of $n \times n$ matrices with entries
that are 0 or 1, for various matrix rings.

In this paper we give an answer to his question when
$R=\Mat_n(\GF[2])$, the $n\times n$ matrices with coefficients in $\GF[2]$, in
the sense of an algorithm to calculate $a(n)$, whose values we give
for $n \le 60$.  We also calculate the closely related sequence $b(n)$
of matrices which are squares in $\GL_n(\GF[2])$.
The calculation is based on the observation that whether or not a matrix $A$ is a square
is a class-function---i.e., it only depends on the conjugacy class of $A$.
We then use the known characterization of conjugacy classes in
$\Mat_n(\GF[2])$ along with formulas for the cardinality of their
centralizers to derive a method for calculating $a(n)$.  Along the way
we also found a number of other interesting sequences which were not
already in the OEIS.

We should
expect that $a(n)$ and $b(n)$ both grow approximately like $2^{n^2}$
because of a general result about word maps \cite{MR2041227}: Let $W$ be a
word in the free group on $d$ generators, $x_1, \dots, x_d$, and $G$ be a group.  We
define a map from the set $G^d = G \times \dots \times G$ ($d$ times)
to $G$, also denoted by $W$, by sending $(g_1, \dots, g_d)$ to the element of $G$ obtained
by substituting $g_i$ for $x_i$ in $W$.  If $W$ is not the trivial
element, then the image $W(G^d)$ is large:  If $G_n$ is a
sequence of non-abelian simple groups such that $|G_n| \rightarrow \infty$, then
\begin{equation*}
  \lim_{n \rightarrow \infty} \log\left(\left|W\left(G^d_n\right)\right|\right)/\log(|G_n|) = 1.
\end{equation*}
To apply this to our case, we take $G_n = \GL_n(\GF[2])$ (which is simple), $d=1$, and
$W=x_1^2$.  This shows that
$\lim_{n\rightarrow  \infty}\log(b(n))/\log(|\GL_n(\GF[2])|)> 0$.
However (see \corref{glorder}) we have $|\GL_n(\GF[2])| \sim
\gamma_2 2^{n^2}$ for an explicit $\gamma_2 > 0$. 
Since $2^{n^2} \ge a(n) \ge b(n)$ we get the above assertion.

Counting squares in simple groups arises in the context of
proving that every element in a simple group can be written as the
product of two squares~\cite{MR2833479}.

\section{Generating and Counting}
\label{sec:counting}

When faced with a problem such as this one, the first thing to try is to
generate all squares and count them.  One can do this is in a
straightforward way: enumerate all matrices of a given size,
square them, and keep track of their counts via a data structure such
as a hash table.  However, one can be much more efficient than that.
Since the matrices that we're interested in have all entries in
$\{0,1\}$ it makes sense to use a Gray code to generate all such
matrices.  There are many different variants of Gray code, but all of
them have the property that adjacent matrices in the sequence
differ in precisely one position.  Suppose that position is $(i,j)$.
Let $E$ denote the matrix which is all 0s except for 1 in the $(i,j)$
position.  We'll keep track of $A$---the current matrix---and
$B = A^2$.  When we change $A$ to $A+E$, we change $B$ to
$(A+E)^2 = A^2 + E A + A E + E^2$.  Note, first, that $E^2 =0$ unless
$i=j$, in which 
case $E^2 = E$.  Note also that $EA$ is all 0s except for its $i$-th row,
which is the $j$-th row of $A$.  Similarly, $AE$ is all 0s except for
its $j$-th column which is the $i$-th column of $A$.  Thus, if we
represent $A$ as $n^2$ bits in a computer word, we can update $B$ by a
few shifts and masking operations, as well as at most 3 exclusive
ORs.  See \lstref{ccode} on page~\pageref{sec:alg} for a C program
implementing this. One can calculate the exact number of $5 \times 5$ matrices
which are squares in about $0.302$ seconds on a fairly standard PC
workstation.  However, this approach can't be pushed much further in
practice, since it needs about $2^{n^2 - 3}$ bytes of storage for a
table of $n^2$ bits and has running time proportional to $2^{n^2}$.
In fact the calculation of $a(6)$ takes 2077.957 seconds on
the same workstation.

\section{Matrices and Their Conjugacy Classes}
\label{sec:matrices}

We begin by collecting the results about conjugacy classes of
matrices over finite fields that we need to derive the algorithm.  The main tool is {\em
  rational canonical form}, which is a generalization of the well-known
Jordan canonical form.  Although this is well known (for example, see
\cite[Chapter XIV]{MR1878556}) few sources\footnote{For example:
  Kung \cite{MR604337} rederived this result independently and Green
  \cite{MR0072878} refers to an unpublished manuscript of Philip Hall,
  which antedates Dickson's paper.}
give the formula for the
order of the centralizer of an element.
For matrices over finite fields this is
originally due to Dickson \cite{MR1507868}.  An exposition in more
modern notation 
may be found in MacDonald \cite[p.~87]{MR553598}.
\begin{definition}
  Two matrices $A,B \in \Mat_n(K)$, where $K$ is a field, are {\em
    similar} (written $A \sim B$) if there is an invertible
  $U \in \Mat_n(K)$ with $A =
  U^{-1} B U$.  A {\em $K$-conjugacy class} is a set of all matrices
  $U^{-1} A U$ for a fixed $A$ and all invertible $U$.  The {\em
    centralizer} of $A$ is $C_K(A) := \{ U \in \GL_n(K) : U^{-1} A U = A
  \}$.

  A matrix $A$ is {\em semisimple} if $A$ is similar to a diagonal
  matrix (over the algebraic closure of the coefficient field).  It is
  {\em semisimple regular} if the elements on the diagonal are
  distinct.
\end{definition}
Note (e.g., see \cite[Chapter XIV]{MR1878556}) that $A$ and $B$
are similar as members of $\Mat_n(K)$ if and only if they are
similar as members of $\Mat_n(L)$ where $L/K$ is any field extension.

Every semisimple matrix $A$ over $\GF[2^n]$ is similar to a diagonal
matrix.  Raising such a matrix to the $2^n$ power permutes the elements
on the diagonal, since $A$ is defined over $\GF[2^n]$.  We have
$A \sim A^{2^n}$ which is obviously a square.
Thus, the set of squares of matrices over $\GF[2^n]$ contains all
semisimple matrices.

If $A,B$ are square matrices, we denote the direct sum by
\begin{displaymath}
  A \oplus B =
  \begin{pmatrix}
    A & 0 \\
    0 & B
  \end{pmatrix}.
\end{displaymath}
If $r$ is a nonnegative integer denote by $[r]A$ the direct sum of
$r$ copies of $A$.

The key observation (from standard group theory) is that the number of
elements in the conjugacy class of $A$ is $|\GL_n(K)|/|C_K(A)|$.

Conjugacy classes have a fixed standard representative given by
combinations of integer partitions and $K$-irreducible polynomials.
\begin{definition}
  A {\em partition} is a non-increasing sequence of nonnegative integers
  $\lambda$, with all but finitely many $\lambda_i = 0$.  Each of the
  $\lambda_i$ is referred to as a {\em part}.
  The weight
  $|\lambda|  = \sum_i \lambda_i$.  The {\em conjugate} of a partition
  $\lambda$, written $\lambda'$, is defined by
  $\lambda'_j := \#\{ i \mid   \lambda_i \ge j\}$.  The {\em multiplicities} of a partition
  $\lambda$ are $m_i(\lambda) = \#\{\lambda_j = i\}$.
  We have $m_i(\lambda) = \lambda'_i - \lambda'_{i+1}$.  Note that
  $|\lambda| = |\lambda'|$ and that
  conjugation is an involution.

  It is convenient to also consider the empty partition, denoted by
  $\emptyset$, with $|\emptyset| = 0$ and $\emptyset' = \emptyset$.

  Let $\cP$ denote the set of partitions.  We also write partitions in
  the form $1^{m_1} 2^{m_2} \dots$, indicating that $i$ occurs with
  multiplicity $m_i$, where we omit those factors with $m_i=0$.
\end{definition}

For a prime power $q$ let $\cI(q)$ denote the set of
$\GF[q]$-irreducible monic polynomials with coefficients in $\GF[q]$,
and $\cI(q)_d$ those elements of $\cI(q)$ of degree $d$.  We also
denote by $\cI'(q)_d = \cI(q)_d$ if $d > 1$ and $\cI'(q)_1 = \{X -
\alpha: \alpha \in \GF[q]^*\}$, i.e., all of $\cI(q)_1$ except for the
polynomial $X$.

\begin{definition}
  Let $\phi$ be a monic polynomial over a field $K$,
  $\phi(X) = \sum_{i=0}^{\deg(\phi)} a_i X^i$.
  Define the companion matrix
  \begin{displaymath}
    M(\phi) = 
    \begin{pmatrix}
      0 & 1 & 0 & \dots & 0 \\
      0 & 0 & 1 & \dots & 0 \\
      \vdots & \vdots & \vdots & \ddots & \vdots \\
      0 & 0 & 0 & \dots & 1 \\
      -a_0 & -a_1 & -a_2 & \dots & -a_{\deg(\phi) - 1}
    \end{pmatrix}.
  \end{displaymath}
  Note that this is the matrix of multiplication by $X$ on polynomials
(mod $\phi(X)$) with respect
  to the ordered basis $1,X,\dots,X^{\deg(\phi)-1}$.
\end{definition}
For convenience, we denote by $M(1)$ the $0 \times 0$ matrix.
\begin{proposition}
  \label{prop:conjugacy}
  Every conjugacy class of an element of $\Mat_n(\mathbb{F}_q)$ is uniquely specified by a function $\Lambda:
  \cI(q) \rightarrow \cP$, where for all but finitely many $\phi \in
  \cI$, we have $\Lambda(\phi) = \emptyset$, and for some $\phi$,
  $\Lambda(\phi) \ne \emptyset$.  The dimension of such a $\Lambda$ is
  $\dim(\Lambda) := \sum_{\phi \in \cI} |\Lambda(\phi)|\deg(\phi)$.
The standard representative of this conjugacy class is
\begin{displaymath}
  \bigoplus_{\phi \in \cI(q)} \bigoplus_i M(\phi^{\Lambda(\phi)_i}).
\end{displaymath}
\end{proposition}
\begin{definition}
  The standard representative for a matrix given by
  \propref{conjugacy} is called the {\em rational canonical form} of
  the matrix.
\end{definition}
Note that the characteristic polynomial of the conjugacy class
represented by $\Lambda$ is $\prod_{\phi \in \cI(q)}
\phi^{|\Lambda(\phi)|}$ and the minimal polynomial is
$\prod_{\phi \in \cI(q)} \phi^{\max_i \Lambda(\phi)_i}$.
\begin{definition}
  Let $M \in \Mat_n(K)$, $\phi_M(X)$ denote its characteristic
  polynomial, and $\psi_M(X)$ denote its minimal polynomial.  We say
  that $M$ is {\em separable} if $\phi_M(X)$ has distinct roots in the
  algebraic closure of $K$, $M$ is {\em semisimple} if $\psi_M(X)$ has
  distinct roots in the algebraic closure, and {\em cyclic} if $\phi_M
  = \psi_M$.

  We note that this definition of semisimple is equivalent to the
  previous definition.  
\end{definition}
\begin{definition}
  The $q$-Pochammer symbol $(a;q)_n := \prod_{k=0}^{n-1}(1-aq^k)$ when
  $n$ is a positive integer.  We also set $(a;q)_0 = 1$ and $(a;q)_{-n} =
  1/(aq^{-n};q)_n$.
\end{definition}

\begin{proposition}[Dickson \cite{MR1507868}]
  \label{prop:centralizer}
  Let $\lambda \in \cP$ be a partition, and $q$ a prime power.
  Define $C(\lambda,q) = q^{\sum_i (\lambda'_i)^2}\prod_{i \ge 1} (q^{-m_i(\lambda)};q)_{m_i(\lambda)}$.
    The order of the centralizer of an element associated to the data
  $\Lambda: \cI(q) \rightarrow \cP$ is $\prod_{\phi \in \cI(q)}
  C(\Lambda(\phi),q^{\deg(\phi)})$. 
\end{proposition}
Given a matrix $A \sim \oplus_iM(\phi(X)^{\lambda_i})$, the vector space of
matrices $U$ such that $U A = A U$ has dimension $\sum_i
(\lambda'_i)^2$.  The quantity $\prod_i(q^{-m_i(\lambda)};q)_{m_i(\lambda)}$ can
be seen as a correction factor to specify that $U$ is invertible.

The identity element corresponds to the data
$\Lambda(X-1) = 1^n$,
and $\Lambda(\phi) = \emptyset$ for $\phi \ne X-1$.  Thus
we have 
\begin{corollary}
  \label{cor:glorder}
  Let $q$ be a prime power.  Then $|\GL_n(q)| = \prod_{k=0}^{n-1} (q^n
  - q^k) = q^{n^2}(q^{-n};q)_n$.
\end{corollary}
Note that $|\GL_n(q)|/q^{n^2} = \prod_{k=1}^{n}(1-q^{-k})$, and the
infinite product\linebreak $\prod_{k=1}^{\infty}(1-q^{-k})$ converges to some
$\gamma_q > 0$.  This shows that $|\GL_n(q)| \sim \gamma_q q^{n^2}$.
The approximate value of $\gamma_2 \approx 0.28878809508660242$.

\section{Powers}
\label{sec:powers}
We now analyze $M(\phi^r)$, where $\phi$ is an irreducible
monic polynomial, and obtain a characterization of squares of
matrices. We start off with a few lemmas which allow us to compute the
effect of raising to the $r$-th power for $r$ a positive integer.

If $\phi(X)$ is a monic polynomial of degree $d$ then $M(X)$ is the matrix of
multiplication by $X$ with respect to the ordered basis
$1,X,\dots,X^{d-1}$ where multiplication of polynomials is taken
(mod $\phi(X)$).  If we write down the matrix of multiplication by $X$
with respect to another basis, then it is similar to $M(X)$.

In the following, if $\phi(X)=(X-\alpha_1)\cdots (X-\alpha_n)$ is a polynomial with roots $\alpha_1,
\dots, \alpha_n$ in the algebraic closure, and $r > 0$ an
integer, denote by $\phi^{(r)}(X)=(X-\alpha_1^r)\cdots (\alpha-\alpha^r_n)$ roots are the
$r$-th powers $\alpha_1^r, \dots, \alpha_n^r$  of the roots of
$\phi$.  If $\phi$ is defined over $\GF[q]$ then so is $\phi^{(r)}$.
When $\phi$ is a polynomial over
$\GF[2]$, we have $\phi^{(2)}(X) = \phi(X)$.

\begin{lemma}[The Chinese Remainder Theorem]
  \label{lem:crt}
  Let $\phi(X)$ and $\psi(X)$ denote relatively prime monic
  polynomials.  Then
  \begin{displaymath}
    M(\phi(X) \psi(X)) \sim M(\phi(X)) \oplus M(\psi(X)).
  \end{displaymath}
\end{lemma}
\begin{proof}
  Since $\phi(X),\psi(X)$ are relatively prime, there are polynomials $u(X), v(X)$
  such that
  \begin{displaymath}
    u(x) \phi(X) + v(X) \psi(X) = 1.
  \end{displaymath}
  Let $d=\deg(\phi), e=\deg(\psi)$, and
  $B_1 := ((X^i v(X)\bmod{\phi(X)})\psi(X): i=0,\dots,d-1)$ and
  $B_2 := ((X^i u(X)\bmod{\psi(X)})\phi(X): i=0,\dots,e-1)$.  Let
  $V_1$ be the span of $B_1$ and $V_2$ the span of $B_2$. Then
  the space spanned by $1,\dots,X^{d+e-1}$ is $V_1 \oplus V_2$, and
  multiplication by $X$ leaves $V_1$ and $V_2$ invariant.
  Furthermore, 
  the matrix of $X$ with respect to $B_1$ is $M(\phi)$ and with
  respect to $B_2$ is $M(\psi)$.
\end{proof}
\begin{corollary}
  \label{cor:split}
  Let $\phi(X)$ be a monic polynomial whose factorization  in the
  algebraic closure is $\phi(X) = \prod_i (X - \alpha_i)^{r_i}$, for
  distinct $\alpha_i$, and $r_i > 0$.  Then
  \begin{displaymath}
    M(\phi(X)) \sim \bigoplus_i \left(\alpha_i I_{r_i} + M(X^{r_i})\right),
  \end{displaymath}
  where $I_{r_i}$ is an $r_i \times r_i$ identity matrix and
  similarity is over the algebraic closure.
\end{corollary}
\begin{proof}
  By \lemref{crt} it suffices to prove the statement when $\phi(X) =
  (X-\alpha)^r$ for some $\alpha \in \overline{\GF[2]}$ and positive
  integer $r$.  However, we have $X = \alpha + (X-\alpha)$, so
  multiplication by $X$ is given by the matrix $\alpha I_r$ plus the
  matrix of multiplication by $(X-\alpha)$.  The latter is obviously
  similar to the matrix of multiplication by $X$ $\bmod\ X^r$.
\end{proof}
\begin{lemma}
  \label{lem:power}
  Let $n,r$ be positive integers.  Then
  \begin{displaymath}
    M(X^n)^r \sim [n \bmod r] M\left(X^{\lceil n/r \rceil}\right) \oplus
    [r-(n \bmod r)] M\left(X^{\lfloor n/r \rfloor}\right).
  \end{displaymath}
\end{lemma}
\begin{proof}
  The matrix $M(X^n)$ is the matrix of multiplication by $X$ modulo
  $X^n$ with respect to the basis $1,X, \dots, X^{n-1}$.  Thus
  $M(X^n)^r$ is the matrix of multiplication by $X^r$ with respect to
  the same basis.   This maps $X^i \mapsto X^{i+r} \mapsto \cdots
  \mapsto X^{i + sr} \mapsto 0$, where
  $s = \lfloor (n-1-i)/r \rfloor$.  When $i < n \bmod r$ we have $s+1 =
  \lfloor (n+r-1)/r \rfloor = \lceil n/r \rceil$, and $\lfloor n/r \rfloor$ otherwise.
\end{proof}
\begin{corollary}
  \label{cor:power2}
  Let $q$ be a power of $2$, $r$ a positive integer, and
   $\phi \in \cI(q)$.  Then $M(\phi(X)^r)^2 \sim
  M(\phi^{(2)}(X)^{\lceil r/2 \rceil}) \oplus M(\phi^{(2)}(X)^{\lfloor r/2 \rfloor}).$
\end{corollary}
\begin{proof}
  Let $\alpha_1, \dots, \alpha_d$ be the roots of $\phi(X)$.
  By \corref{split} we have
  \begin{displaymath}
    \begin{aligned}
      M(\phi(X)^r)^2 & \sim \bigoplus \left(\alpha_i I + M(X^r)\right)^2
       = \bigoplus_{i=1}^d \left(\alpha_i^2 I + M(X^r)^2\right) \\
       & = \bigoplus_{i=1}^d \left(\alpha_i^2 I + M(X^{\lfloor r/2 \rfloor})\right)
      \oplus \bigoplus_{i=1}^d \left(\alpha_i^2 I + M(X^{\lceil r/2 \rceil})\right)\\
      & \sim M\left(\phi^{(2)}(X)^{\lceil r/2 \rceil}\right) \oplus M\left(\phi^{(2)}(X)^{\lfloor r/2 \rfloor}\right).
    \end{aligned}
  \end{displaymath}
  The first equality follows because the characteristic is 2.
  The second equality follows from \lemref{power}, and the last line
  from \corref{split}.
\end{proof}
\begin{corollary}
  \label{cor:image}
  If a matrix $A \in \Mat_n(\GF[q])$, where $q$ is a power of $2$, is in
  the conjugacy class with standard representative specified by $\Lambda :
  \cI(q) \rightarrow \cP$, then the conjugacy class containing $A^2$
  has its standard representative specified by
  $\cM :
  \cI(q) \rightarrow \cP$, where $m_i(\cM(\phi)) = 2 m_{2i}(\Lambda(\phi^{(2)}))
  + m_{2i-1}(\Lambda(\phi^{(2)})) + m_{2i+1}(\Lambda(\phi^{(2)}))$
 for all $\phi$ and $i \ge 1$.
\end{corollary}
\begin{proof}
  If
\begin{displaymath}
  A \sim \bigoplus_{\phi \in \cI(q)} \bigoplus_j M\left(\phi(X)^{\Lambda(\phi)_j}\right),
\end{displaymath}
then
\begin{displaymath}
  A^2 \sim \bigoplus_{\phi \in \cI(q)} \bigoplus_j M\left(\phi^{(2)}(X)^{\Lambda(\phi)_j}\right)^2.
\end{displaymath}
It thus suffices to show that for all $\phi \in \cI(q)$, and $\lambda
\in \cP$
\begin{displaymath}
  \bigoplus_j M\left(\phi(X)^{\lambda_j}\right)^2 \sim \bigoplus_j M\left(\phi^{(2)}(X)^{\mu_j}\right),
\end{displaymath}
where $\mu \in \cP$ satisfies $m_i(\mu) = 2 m_{2i}(\lambda)
+ m_{2i-1}(\lambda) + m_{2i+1}(\lambda)$.
By \corref{power2} we have
\begin{displaymath}
  M\left(\phi(X)^{\lambda_j}\right)^2 \sim  M\left(\phi^{(2)}(X)^{\lceil \lambda_j/2 \rceil}\right)
  \oplus M\left(\phi^{(2)}(X)^{\lfloor \lambda_j/2 \rfloor}\right).
\end{displaymath}
If $\lambda_j = 2i$ then it contributes 2 to the multiplicity
$m_i(\mu)$.  Otherwise it contributes 1, thus giving the assertion.
\end{proof}
By enumerating all partitions of $n$ one can produce the set of
partitions of the form $\cM(\phi)$ as in the above corollary.  Counting
these partitions produces the sequence
$1,1,2,3,4,5,7,10,\allowbreak 13,16,21,28,35,43,55,70, \dots$ which is 
sequence A006950 in the OEIS.  There are many comments there about
classes of partitions.  They include: ``Also the
number of partitions of $n$ in which all odd parts occur with
multiplicity 1.''  None of the above partitions fell
into any of the classes referred to, but the conjugates did.  This
yielded \propref{inverse} below.  The sequence above is
also the sequence of coefficients of one of Ramanujan's mock
$\vartheta$ functions whose
generating function is
\begin{displaymath}
  \vartheta(z) := \prod_{k > 0} \frac{1+z^{2k-1}}{1-z^{2k}}.
\end{displaymath}
\begin{proposition}
  \label{prop:inverse}
  Let $\Delta : \cP \rightarrow \cP$, be defined by
  $m_i(\Delta(\lambda)) = 2 m_{2i}(\lambda) + m_{2i-1}(\lambda) +
  m_{2i+1}(\lambda)$, for all $i$.
  Then $\mu$ is in the image of
  $\Delta$ if and only if $m_{2i-1}(\mu') \le 1$ for $i \ge 1$.
\end{proposition}
\begin{proof}
  Let $\mu = \Delta(\lambda)$. We have
  $\mu'_i = \sum_{j \ge i} m_j(\mu)$.
  Substituting in the value of $m_j(\mu)$ we obtain
  \begin{displaymath}
    \mu'_i = \sum_{j \ge i}  (2 m_{2j}(\lambda) + m_{2j-1}(\lambda) +
    m_{2j+1}(\lambda)) \equiv m_{2i-1}(\lambda) \pmod{2}.
  \end{displaymath}
  If $m_j(\mu') \ge 2$ then there is an $i$ such that
  $j = \mu'_{i+1} = \mu'_i$. We have $m_i(\mu) =
  \mu'_i - \mu'_{i+1} =
  0$ and thus $m_{2i-1}(\lambda) = 0$.  The above congruence shows that
  $j = \mu'_i$ is even.
  In other words $m_j(\mu') \le 1$ if
  $j$ is odd.
  
  For the converse, suppose that we have a partition
  $\mu$ so that $m_{2i-1}(\mu') \le 1$ for all $i$.
  Define the sequence $b_i$ as follows: for all $i\ge 1$ set $b_{2i-1}
  \in \{0,1\}$, such that $b_{2i-1} \equiv \mu'_i \pmod{2}$.  For
  all $i\ge 1$ set $b_{2i} = (m_i(\mu) - b_{2i-1} - b_{2i+1})/2$.
  First, we show that $b_{2i} \in \ZZ$. We have $m_i(\mu) = \mu'_i
  - \mu'_{i+1}$.  Thus $m_i(\mu) - b_{2i-1} - b_{2i+1} \equiv
  (\mu'_{i+1} + b_{2i+1}) + (\mu'_i + b_{2i-1}) \equiv 0 \pmod{2}$.
  Second, we show that $b_{2i} \ge 0$.  This is trivially true if
  $m_i(\mu) \ge 2$, since $b_{2i-1},b_{2i+1} \in \{0,1\}$.
  If $m_i(\mu) = 1$, since $m_i(\mu) \equiv b_{2i-1} + b_{2i+1}
  \pmod{2}$, not both of the $b$ can be 1.   If
  $m_i(\mu) = 0$, then $\mu'_{i+1} = \mu'_i$, which, by the condition
  on $\mu$ implies that $\mu'_i$ and $\mu'_{i+1}$  are even.  By the
  construction above this implies that $b_{2i-1} = b_{2i+1} = 0$, and
  thus $b_{2i} = 0$.  We then construct a
  partition $\lambda$ such that $m_i(\lambda) = b_i$.  This exists 
  since the only necessary condition on $b_i$ for the existence of
  such a partition is that $b_i \ge 0$ and $b_j = 0$ for $j$
  sufficiently large.
\end{proof}

The key result in allowing an efficient calculation of our sequences
is that the conjugacy classes involved are exactly those with a
restriction on the possible partitions, but any irreducible polynomial
is allowed.

\begin{definition}
  For each $\phi \in \cI(q)$ let $\cS_{\phi} \subseteq \cP$ be a subset
  of partitions containing the empty partition.
  
  We call a sequence $\cC_n \subseteq \Mat_n(\GF[q])$ of a union of
  conjugacy classes {\em partition restricted} with respect to the
  family $\{\cS_{\phi}\}$ if the functions
  $\Lambda : \cI(q) \rightarrow
  \cP$ which describe the elements of $\cC_n$ are exactly the
  functions 
  such that $\dim \Lambda = n$ and, for all $\phi \in \cI(q)$, we have
  $\Lambda(\phi) \in \cS_{\phi}$.  If all $\cS_{\phi}$ are the same (in
  which case we drop the subscript) we call the sequence $\cC_n$ 
  \emph{partition uniform} with respect to~$\cS$.
\end{definition}

Using these results yields 
\begin{theorem}
  Let $\cC_n \subseteq \Mat_n(\GF[2])$ denote the set of squares of
  elements of $\Mat_n(\GF[2])$.  Then $\cC_n$ is a union of conjugacy
  classes and it is partition uniform with respect to $\cS = \{
  \lambda \in \cP : m_{2i-1}(\lambda') \le 1, i \ge 1\} \cup \{\emptyset\}$.  
\end{theorem}
\begin{proof}
  If a conjugacy class is specified by $\lambda: \cI(2) \rightarrow
  \cP$ its standard representative is
  \begin{displaymath}
    \bigoplus_{\phi \in \cI(2)} \bigoplus_i M\left(\phi^{\lambda(\phi)_i}\right).
  \end{displaymath}
  and thus its square is conjugate to
  \begin{displaymath}
    \bigoplus_{\phi \in \cI(2)} \bigoplus_i M\left(\phi^{\lambda(\phi)_i}\right)^2.
  \end{displaymath}
  Thus it suffices to consider
  $M(\phi(X)^r)^2$, where $\phi$ is irreducible.  The proof is
  finished using \corref{image} and \propref{inverse}.
\end{proof}

Armed with the above characterization of conjugacy classes of squares,
one could proceed by enumerating all such classes, and then summing
the sizes of the conjugacy classes, to get the desired counts.  This,
indeed, would adhere to Wilf's dictum\footnote{More precisely, if $\cA_n$ denotes a class of combinatorial
  object of ``size'' $n$, a good answer would be an algorithm to
  calculate $|\cA_n|$ with running time $o(|\cA_n|)$.}
 \cite{MR0653502} of a ``good answer'', since the
number of such classes appears to be of the order of $2^n$.  In fact,
as we shall see in \thmref{exact}, they are precisely of this order.   The first
60 terms of the
sequence of the number of such classes appears in \tabref{conj}.
However we can do much better, as we shall see in the next section.

\begin{table}[h!]
\caption{Conjugacy classes of squares}
  \label{tab:conj}
  \centering
  \tiny
  \begin{tabular}[htbp]{|l||r|r|}
    \hline 
    $n$ & \multicolumn{1}{c|}{$\GL_n(\GF[2])$} & \multicolumn{1}{|c|}{$\Mat_n(\GF[2])$} \\
    \hline
    $1$ & $1$ & $2$ \\
    $2$ & $2$ & $4$ \\
    $3$ & $5$ & $10$ \\
    $4$ & $10$ & $22$ \\
    $5$ & $20$ & $46$ \\
    $6$ & $41$ & $96$ \\
    $7$ & $82$ & $198$ \\
    $8$ & $166$ & $406$ \\
    $9$ & $334$ & $826$ \\
    $10$ & $667$ & $1668$ \\
    $11$ & $1336$ & $3362$ \\
    $12$ & $2682$ & $6770$ \\
    $13$ & $5360$ & $13590$ \\
    $14$ & $10724$ & $27248$ \\
    $15$ & $21467$ & $54614$ \\
    $16$ & $42936$ & $109378$ \\
    $17$ & $85876$ & $218946$ \\
    $18$ & $171786$ & $438180$ \\
    $19$ & $343574$ & $876738$ \\
    $20$ & $687184$ & $1753998$ \\
    $21$ & $1374427$ & $3508726$ \\
    $22$ & $2748852$ & $7018368$ \\
    $23$ & $5497766$ & $14038006$ \\
    $24$ & $10995706$ & $28077846$ \\
    $25$ & $21991402$ & $56157954$ \\
    $26$ & $43982908$ & $112318900$ \\
    $27$ & $87966150$ & $224642090$ \\
    $28$ & $175932383$ & $449289666$ \\
    $29$ & $351864964$ & $898586438$ \\
    $30$ & $703730584$ & $1797182704$ \\
    $31$ & $1407461288$ & $3594378014$ \\
    $32$ & $2814923196$ & $7188772666$ \\
    $33$ & $5629847656$ & $14377567834$ \\
    $34$ & $11259695532$ & $28755164100$ \\
    $35$ & $22519392276$ & $57510365698$ \\
    $36$ & $45038787489$ & $115020782350$ \\
    $37$ & $90077575358$ & $230041628622$ \\
    $38$ & $180155153036$ & $460083340304$ \\
    $39$ & $360310311906$ & $920166792942$ \\
    $40$ & $720620625522$ & $1840333728182$ \\
    $41$ & $1441241255486$ & $3680667639522$ \\
    $42$ & $2882482522524$ & $7361335523444$ \\
    $43$ & $5764965048250$ & $14722671356642$ \\
    $44$ & $11529930107318$ & $29445343113738$ \\
    $45$ & $23059860237589$ & $58890686756910$ \\
    $46$ & $46119720481194$ & $117781374180336$ \\
    $47$ & $92239440983766$ & $235562749221166$ \\
    $48$ & $184478882017076$ & $471125499580570$ \\
    $49$ & $368957764045976$ & $942251000588770$ \\
    $50$ & $737915528134398$ & $1884502003008980$ \\
    $51$ & $1475831056367066$ & $3769004008432714$ \\
    $52$ & $2951662112765356$ & $7538008019902670$ \\
    $53$ & $5903324225614736$ & $15076016043685054$ \\
    $54$ & $11806648451425570$ & $30152032092453552$ \\
    $55$ & $23613296902912949$ & $60304064191298614$ \\
    $56$ & $47226593806008646$ & $120608128390767918$ \\
    $57$ & $94453187612408280$ & $241216256792193274$ \\
    $58$ & $188906375224938380$ & $482432513597744820$ \\
    $59$ & $377812750450241204$ & $964865027212545410$ \\
    $60$ & $755625500901295794$ & $1929730054447325946$ \\
    \hline 
  \end{tabular}
\end{table}

\section{Generating Functions}
\label{sec:generating}
Let $\cC_n \subseteq \Mat_n(\GF[2])$ be a union of conjugacy classes
for each $n \ge 1$.
We associate two generating functions with $\cC$.  The first has
coefficients which give the numbers of elements in the conjugacy class
$\cC_n$ (scaled by the total number of invertible elements):
$$F_{\cC}(x) := 1 + \sum_{n=1}^{\infty} \frac{|\cC_n|}{|\GL_n(\GF[2])|} x^n.$$
The second has coefficients which give the number of conjugacy classes
in $\cC_n$:
$$G_{\cC}(x) := 1 + \sum_{n=1}^{\infty} |\#\{ \Lambda \in \cC_n \}|x^n,$$
where, by abuse of notation, we say that $\Lambda \in \cC_n$ if
$\Lambda\colon \cI(q) \rightarrow \cP$ specifies a conjugacy class in $\cC_n$.

We use the coefficient $|\cC_n|/|\GL_n(\GF[2])|$ because it is
\begin{equation}
\label{eq:decomp}
\sum_{\Lambda} \frac{1}{C(\Lambda,q)}
\end{equation}
where the outer sum is taken over all
$\Lambda\colon \cI(q) \rightarrow \cP$, specifying the conjugacy classes in
$\cC_n$, and because 
$C(\Lambda,q)$ has a multiplicative decomposition as in \propref{conjugacy}.

The reason for the definitions of partition restricted and partition
uniform is the following:
\begin{proposition}
  Let $\cC_n \subseteq \Mat_n(\GF[q])$ be a union of conjugacy classes,
  and $F_{\cC}(X)$ the associated generating function
  (resp., $G_{\cC}(X)$ is the associated generating function for the
  number of conjugacy classes).  If $\cC_n$ is
  partition restricted with respect to $\cS_{\phi}$ then
  \begin{equation}
    \label{eq:product}
    F_{\cC}(X) = \prod_{\phi \in \cI(q)} \sum_{\lambda \in \cS_{\phi}}
    \frac{1}{C(\lambda,q^{\deg(\phi)})} X^{|\lambda| \deg(\phi)},
  \end{equation}
  and
  \begin{equation}
    \label{eq:product:conj}
    G_{\cC}(X) = \prod_{\phi \in \cI(q)} \sum_{\lambda \in \cS_{\phi}}
     X^{|\lambda| \deg(\phi)}.
  \end{equation}
  If $\cC_n$ is partition uniform with respect to $\cS$ then 
  \begin{equation}
    \label{eq:uniform}
    F_{\cC}(X) = \prod_{d=1}^{\infty} \left(\sum_{\lambda \in \cS}
      \frac{1}{C(\lambda,q^d)} X^{|\lambda|d}\right)^{|\cI(q)_d|},
  \end{equation}
  and
  \begin{equation}
    \label{eq:uniform:conj}
    G_{\cC}(X) = \prod_{d=1}^{\infty} \left(\sum_{\lambda \in \cS}
    X^{|\lambda|d}\right)^{|\cI(q)_d|}.
  \end{equation}
\end{proposition}
\begin{proof}
  We use the multiplicative decomposition from \propref{centralizer}
  to see that
  \begin{displaymath}
    F_{\cC}(X) = \sum_{\Lambda} \prod_{\phi \in \cI(q)}
    \frac{1}{C(\Lambda(\phi),q^{\deg(\phi)})} X^{|\Lambda(\phi)| \deg(\phi)},
  \end{displaymath}
  where the sum is over all possible $\Lambda$ describing the
  conjugacy classes in $\cC_n$.  By the definition of partition restricted, we
  may interchange the summation and product obtaining Equation~\eqref{eq:product}.
  By definition of partition uniform we then have
  Equation~\eqref{eq:uniform}.  A similar argument applies to $G_{\cC}$.
\end{proof}

If we are interested only in classes in $\GL_n(\GF[2])$ instead of
$\Mat_n(\GF[2])$ we modify Equations~\eqref{eq:uniform} and~\eqref{eq:uniform:conj} by using the
exponent $q-1$ instead of $|\cI(q)_1|= q$, which corresponds to omitting
the irreducible polynomial $\phi(X) = X$.

We now show that the number of conjugacy classes of squares (both for
all matrices and for invertible matrices) grows exactly as $2^n$.
\begin{lemma}
  \label{lem:product}
  Let $q$ be a prime power.  As formal power series we have
  \begin{equation}
    \label{eq:product:formula}
    1-qX = \prod_{n=1}^{\infty} (1-X^n)^{|\cI(q)_n|}.
  \end{equation}
\end{lemma}
\begin{proof}
  Since both the left- and right-hand sides of
  Equation~\eqref{eq:product:formula} have constant term 1, it
  suffices to show that the logarithmic derivatives of both sides are
  equal.  The logarithmic derivative of the right-hand side is
  \begin{displaymath}
    \begin{aligned}
     -\sum_{n=1}^{\infty} \frac{n|\cI(q)_n|X^{n-1}}{1-X^n} 
      = & -\sum_{n=1}^{\infty} \sum_{j=0}^{\infty}
      n|\cI(q)_n|X^{(j+1)n-1} \\
      = & -\sum_{m=1}^{\infty} X^{m-1} \sum_{d|m} d | \cI(q)_d| \\ = &
      -\frac{1}{X} \sum_{m=1}^{\infty} (qX)^m = -\frac{q}{1-qX},
    \end{aligned}
  \end{displaymath}
  which is the logarithmic derivative of the left-hand side.  In the
  above we have used the fact that $\sum_{d|m}d | \cI(q)_d| = q^m$.
  This holds because every element of $\GF[q^m]$ is the root of some
  irreducible polynomial over $\GF[q]$ of degree $d\mid m$ (and
  conversely), and each of those polynomials has exactly $d$ roots.
\end{proof}
\begin{theorem}
  \label{thm:exact}
  Let $a'(n)$ denote the number of conjugacy classes of
  squares for $\Mat_n(\GF[2])$ and $b'(n)$ the number of conjugacy
  classes of squares
  for $\GL_n(\GF[2])$.  We have
  \begin{align*}
    1 + \sum_{n=1}^{\infty} a'(n) z^n & = \prod_{n=1}^{\infty}\frac{1-2z^{2n}}{(1-2z^n)(1-2z^{4n})}\\
\intertext{and}
    1 + \sum_{n=1}^{\infty} b'(n) z^n & = \prod_{n=1}^{\infty}\frac{(1-z^{2n})(1-2z^{2n})}{(1+z^{2n-1})(1-2z^n)(1-2z^{4n})}.
  \end{align*}
  From this it follows that
  there are real $\alpha', \beta' > 0$ such
  that $a'(n) \sim \alpha' 2^n$ and $b'(n) \sim \beta' 2^n$.  
\end{theorem}
\begin{proof}
  When $\cC$ is the set of conjugacy classes of squares of all
  matrices then
  $G_{\cC}(z) = \prod_{d=1}^{\infty} \vartheta(z^d)^{\cI(q)_d}$.  When
  we are dealing with invertible matrices the only factor that differs
  is the one for $d=1$.  Since $|\cI(q)'_1| = |\cI(q)_1|-1$, we must
  divide the above by $\vartheta(z)$.  However we have
  \begin{displaymath}
    \begin{aligned}
      \vartheta(z) & = \prod_{n=1}^{\infty} \frac{1+z^{2n-1}}{1-z^{2n}}
      \\ & = \prod_{n=1}^{\infty}\frac{(1-z^{2n})}{(1-z^n)(1-z^{4n})}.
    \end{aligned}
  \end{displaymath}
  We now apply \lemref{product} to each of the factors and find that
  \begin{displaymath}
    G_{\cC}(z) = \prod_{n=1}^{\infty}\frac{1-2z^{2n}}{(1-2z^n)(1-2z^{4n})}.
  \end{displaymath}
  This product clearly converges when $|z| < 1/2$, has a simple
  pole at $z=1/2$, and has no other singularities when $|z| = 1/2$.  We have
  \begin{displaymath}
    \alpha' := \lim_{z \rightarrow 1/2} (1-2z)G_{\cC}(z) = \prod_{n=1}^{\infty}\frac{1-2(1/2)^{2n}}{(1-(1/2)^n)(1-2(1/2)^{4n})}.
  \end{displaymath}
  Thus, $a'(n) \sim \alpha' 2^n$.
  For the case that $\cC$ specifies invertible squares we must divide
  the above by $\vartheta(1/2)$.
\end{proof}

We may calculate $F_{\cC}(X)$ by using any of the standard fast algorithms
for manipulating power series \cite{MR0520733} but we may exploit its
special form for a more efficient calculation as follows.

A large part of the calculation involves calculating a product
\begin{displaymath}
  F(X) = \prod_{d=1}^n f_d(X)^{n_d}
\end{displaymath}
for power series $f_d(X)$ whose constant term is 1, and positive
integer exponents $n_d$.  We may speed up this calculation
substantially as follows:  taking the logarithmic derivative, we have
\begin{displaymath}
  \frac{F'(X)}{F(X)} = \sum_{d=1}^n n_d \frac{f_d'(X)}{f_d(X)},
\end{displaymath}
of which we're interested in the first $n+1$ terms.  Treating the
coefficients of $F(X)$ after the constant term (which is 1) as
unknowns, we get a linear system by multiplying both sides by $F(X)$
and equating coefficients.  In fact, the linear system is lower
triangular, and so may be solved quickly.  For large $n$ we
may do this more quickly by using the algorithm described in \cite{MR2527719}.
Making this change sped up
the calculation for $n=14$ from 318 seconds to 1 second.  This speed-up improves substantially for larger $n$.

As an alternative to directly manipulating the coefficients as large rational
numbers, we may use the Chinese Remainder Theorem.  Choose distinct
odd primes $p_1, \dots, p_r$ so that $\prod_i p_i > 2^{n^2}$, and
$p_i$ does not divide $2^k -1$ for $k \le n$.  Note that by the prime
number theorem (or weaker estimates) we may do this with $p_i \approx
\log(2^{n^2}) = n^2 \log(2)$, and $r \approx n^2/ \log(n)$. We then have
$C(\lambda,q^d) \not \equiv 0 \pmod{p_i}$, so that we may calculate
the truncated power series of each of the above summands modulo
$p_i$, and then the truncated version of $F_{\cC}$ modulo each of the
$p_i$.  We then multiply the coefficients of $x^k$ by the
$|\GL_k(\GF[2])| \bmod {p_i}$, and finally use the Chinese Remainder
Theorem to recover $|S_k|$ since we know that it is a positive integer
$\le 2^{n^2}$.

Note that the proof of \propref{inverse} yields an algorithm to decide
whether or not a matrix
in $\Mat_n(\GF[2])$ is a square, and, if so, calculate a square root.
Namely, using algorithms for rational canonical form \cite{MR1484489} we
obtain a change-of-basis matrix and a standard representative.  We
use the construction in the proof of \propref{inverse} to find a
partition associated with the class of a square root.  Finally, we
use the change-of-basis matrix to transform the rational canonical
form of the square root.

\section{Results}
\label{sec:results}

We programmed the algorithm described above in the SAGE system for
symbolic calculation \cite{sage} and used it
to calculate the first 60
terms of the following sequences:

\begin{table}[h!] 
\caption{Calculation Times}
  \label{tab:calc}
 \begin{center}
  \begin{tabular}{|r|r|r|}
    \hline
    Classes & Partition Count & Time \\
    \hline \hline
    Squares in $\Mat_n(\GF[2])$ & 641800 & 429.69 sec \\
    Squares in $\GL_n(\GF[2])$ & 157671 & 99.30 sec \\
    Separable elements in $\Mat_n(\GF[2])$ & 1 & 1.11 sec \\
    Separable elements in $\GL_n(\GF[2])$ & 1 & 1.10 sec \\
    Semisimple elements in $\Mat_n(\GF[2])$ & 60 & 1.08 sec \\
    Semisimple elements in $\GL_n(\GF[2])$ & 60 & 1.35 sec \\
    \hline 
  \end{tabular}
  \end{center}
\end{table}
In \tabref{m} (pages~\pageref{begin20437s}--\pageref{end20437s}) we give the first 60 terms of the sequence $a(n)$, the number of $n \times n$ matrices with
coefficients in $\GF[2]$ which are squares of other such matrices.  
The related sequence $b(n)$ in which the matrices are invertible is
given in~\tabref{gl} (pages~\pageref{begin20498s}--\pageref{end20498s}).  We generated these tables in
about 200 seconds each on a workstation.

In order to show that there is an $\alpha > 0$ such that $a(n) \sim
\alpha 2^{n^2}$ (and similarly for $b(n)$ and $\beta$) it would suffice to show that $F_{\cC}(z)$ is holomorphic
in the disk $\{z \in \CC: |z| \le 1 + \epsilon\}$ apart from having a simple pole
at $z=1$ with residue $-\alpha/\gamma_2$, where $\gamma_2 :=
\prod_{n-1}^{\infty} (1-2^{-n}) \approx 0.28878809508660242$ (since $|\GL_n(\GF[2])| \sim \gamma_2
2^{n^2}$).
We conjecture that this is, indeed, the case.
Note that Wall \cite{MR1711918} has proved similar statements  
when $\cC$ specifies the classes of semisimple, regular, and
regular semisimple matrices over a finite field.

As a sanity check on the conjecture we have calculated the
coefficients of the first 71 coefficients $c_0, \dots, c_{70}$ of
$F_{\cC}(z)$, where $\cC$ is the classes of squares of invertible
matrices, within an
accuracy of $2^{-3600}$, and set
$\widehat{\beta} = c_{70}$ to be the coefficient of $z^{70}$.  We
plot below $|c_j - \widehat{\beta}|^{-1/j}$
for $j=1,\dots,69$.
We have $\widehat{\beta} \approx 0.5844546428649343516383$.

\begin{figure}[h!]
  \begin{center}
    \includegraphics[width=14cm]{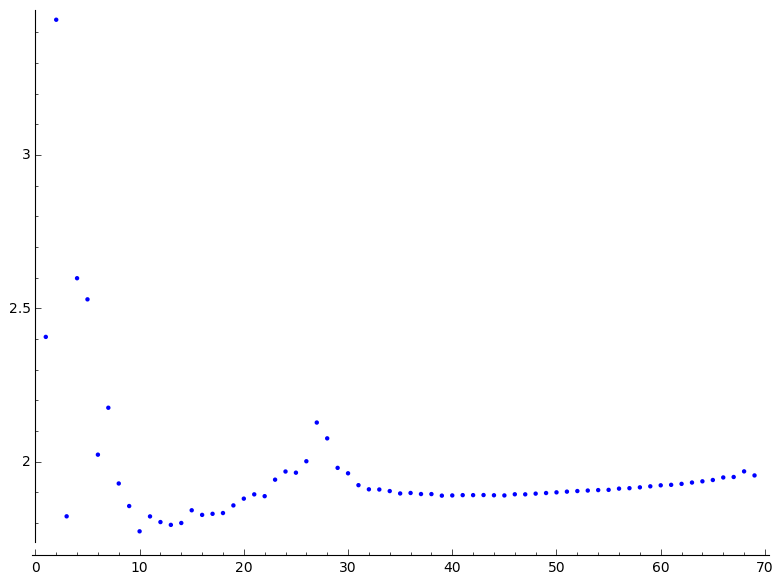}
    \end{center}
  \caption{$|c_j - \widehat{\beta}|^{-1/j}$}
  \label{fig:glratio}
\end{figure}
From~\figref{glratio}, it appears that, except for a pole at $z=1$,
$F_{\cC}(z)$ is holomorphic in an open disk of radius
$1.954579780196859562\dots$ centered at~0.
We give a similar plot in \figref{mratio}, where $\cC$ is the class of squares of all
matrices, and the coefficients of $F_{\cC}(z)$ are $d_0, d_1, \dots$.

\begin{figure}[h!]
  \begin{center}
    \includegraphics[width=14cm]{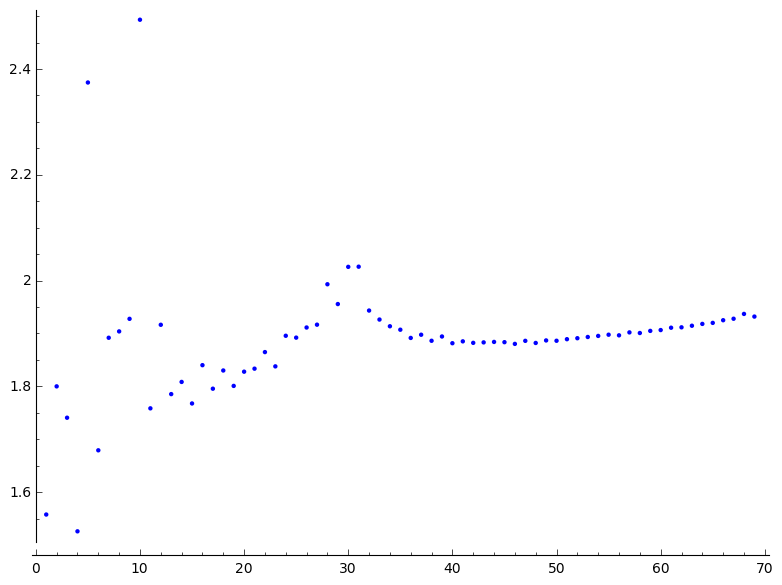}
    \end{center}
  \caption{$|d_j - \widehat{\alpha}|^{-1/j}$}
  \label{fig:mratio}
\end{figure}
From these figures, it appears that
$\widehat{\alpha} \approx
1.358036747413654505$, and that, apart from the pole at $z=1$
that $F_{\cC}(z)$ is holomorphic in an open disk of radius $1.931991705356004184580743154\dots$.
\section{Remarks and Open Problems}
\label{sec:remarks}
The decomposition given in Equation~\eqref{eq:product} is closely
related to the product formula for the cycle index as described in
\cite{MR1696313,MR604337,MR937520}.  Following \cite{MR937520}, we define a
generating function
\begin{displaymath}
  Z_d(q;x) := \frac{1}{[q]_d} \sum_{\alpha \in
    \Mat_d(\GF[q])} \prod_{\phi,\lambda}
  x_{\phi,\lambda}^{a_{\phi,\lambda}(\alpha)},
\end{displaymath}
where $\phi \in \cI(q), \lambda \in \cP$, $a_{\phi,\lambda} = 1$
if $(\phi,\lambda)$ occurs in the description of the conjugacy class
of $\alpha$ and $0$ if it does not, and $[q]_d = |\GL_d(\GF[q])|$.
Define $Z_0(q;x) = 1$.
We
then define a generating function
\begin{displaymath}
  \Phi(u;x) := 1 + \sum_{d \ge 1} Z_d(q;x) u^d.
\end{displaymath}
Kung~\cite{MR604337} and Stong~\cite{MR937520} prove
the factorization
\begin{displaymath}
  \Phi(u;x) = 
  \prod_{\phi \in \cI(q)} \left(1 +  \sum_{\lambda \in \cP}
    \frac{x_{\phi,\lambda}}{C(\lambda,\phi)} u^{|\lambda|d}\right).
\end{displaymath}
If our sets $\cC_n$ are partition restricted, and we set
$x_{\phi,\lambda}=1$ if $\lambda \in \cS_{\phi}$ and 0 otherwise, then
we recover Equation~\eqref{eq:product}.

Fulman
\cite{MR1696313} and Lehrer \cite{MR1329462} have given closed-form expression for the ratios
corresponding to $\alpha$ and $\beta$ in
the three cases treated by Wall---the semisimple, regular, and
regular semisimple matrices.  We leave it as an open problem to prove
the conjecture in \secref{results} and find a closed-form expression
for $\alpha$ and~$\beta$.

As a generalization of \propref{inverse} it would be interesting to
find a closed-form expression for the number of partitions $\lambda$
with the property $m_i(\mu) = 2m_{2i}(\lambda) + m_{2i-1}(\lambda) +
m_{2i+1}(\lambda)$ for all $i$ when $m_{2i+1}(\mu') \le 1$.  This amounts
to counting the number of integer points in the polytope given by the
above equalities and $m_i(\lambda) \ge 0$.

The methods used here should, in principle, be able to be used to
answer analogous questions, such as the number of $n\times n$ matrices
over $\GF[3]$ which are cubes.

\section{Acknowledgments}
\label{sec:ack}

First, and foremost, I'd like to thank Neil Sloane for posing the
problem, and for his magnificent creation of the OEIS.  It has proven
invaluable to me and to countless other mathematicians.  I'd also like
to thank Bob Guralnick for referring me to his work and others' on
related problems, and Jason Fulman for useful correspondence.
Last, I'd like to thank the late Herb Wilf whose wonderful works have been
an inspiration to me.
\appendix

\section{Algorithms}
\label{sec:alg}
\lstlistoflistings
\begin{ccode}[caption={C program for exhaustively calculating $a(n)$},label=lst:ccode]
#include <stdlib.h>
#define BITS_PER_BYTE 8
inline void SetBit(unsigned long int *t,const unsigned long int i) {
  const int bits = BITS_PER_BYTE*sizeof(unsigned long int);
  t[i/bits] |= 1L << (i%bits);
}
unsigned long int CountUp(unsigned long int *tab,const long int n) {
  unsigned long int count = 0UL;
  long int i;
  for(i=0; i < n; i++) {
    count += __builtin_popcountl(tab[i]);
  }
  return count;
}
unsigned long int SloaneExhaust(const int n) {
  const long int n2 = 1L<<(n*n);
  unsigned long int *table;
  const long int lchunk = sizeof(unsigned long int);
  const long int chunk = BITS_PER_BYTE*lchunk;
  const int t_size = (n2 + chunk - 1)/chunk;
  int t;
  unsigned long int A,A2;
  unsigned long int rmask,cmask;
  unsigned long int bigcount;
  table = (unsigned long int *)calloc(t_size,sizeof(unsigned long int));
  A = 0L; B = 0L;
  /* $B$ reprensents the current value of $A^2$ */
  SetBit(table,B);
  rows_mask = (1L<<n) - 1L;
  col_mask = 0L;
  for(t=0; t < n*n; t += n) {
    col_mask |= 1L<<t;
  }
  for(t=1; t < n2; t++) {
    /* get bit index to be flipped in Gray code */
    const long int k = __builtin_ctzl(t);
    const long int i = k/n;
    const long int j = k%n;
    const unsigned long int Eij = 1L<<k;
    /* Eij represents $E_{i,j}$ */
    /* $B \leftarrow B + A E + E A$ */
    B ^= (((A>>(n*j))&row_mask)<<(n*i)) ^ (((A>>i)&col_mask)<<j);
    if (i == j)
      B ^= Eij; /* $B \leftarrow B + E_{i,j}$ */
    A ^= Eij; /* $A \leftarrow A + E_{i,j}$ */
    SetBit(table,B);
  }
  bigcount = CountUp(table,t_size);
  free(table);
  return bigcount;
}
\end{ccode}
\clearpage
\begin{pycode}[caption={SAGE programs for improved algorithm},label=lst:mainsage]
def Count(n,all=True):
    return CountGeneral(n,myparts,all=all)

def CountGeneral(n,s,all=True):
    R.<x> = PowerSeriesRing(QQ,default_prec = n+1)
    pp = [R(1) for _ in range(n+1)]
    for w in range(1,n+1):
        for lam in s(w):
            for d in range(1,n//w+1):
                pp[d] += centralizer(lam,QQ(2)^d)^-1*x^(w*d)
    res = BigProduct(pp[1:],[Irr(2,d,all=all) for d in range(1,n+1)],n)
    return [GLorder(_+1,2)*res[_] for _ in range(n)]

# auxilliary routines
def multiplicity(p):
    # p is a partition
    l = p.conjugate().to_list()
    return [l[_-1] - l[_] for _ in range(1,len(l))] + [l[-1]]

def qPoch(a,q,n):
    return prod([1-a*q^k for k in range(n)])

def centralizer(lam,q):
    if lam.size() == 0:
        return 1
    return q^sum([x^2 for x in lam.conjugate()])*prod([qPoch(q^(-x),q,x) 
      for x in multiplicity(lam)])

def GLorder(n,q):
    return centralizer(Partition(n*[1]),q)

def Irr(q,d,all=True):
    if d == 1 and (not all):
        return q-1
    return sum([moebius(d//n)*(q^n) for n in divisors(d)])/d

def LDeriv(f):
    return f.derivative()/f

# routine for calculating $\prod_d f_d(X)^{n_d}$
def BigProduct(l,p,prec):
    # l are power series, and p are exponents
    n = len(l)
    xx = sum([p[_]*LDeriv(l[_]) for _ in range(n)])
    R = xx.parent().base_ring()
    res = vector(R,xx.padded_list()[:prec])
    mat = matrix(R,prec,prec)
    for k in range(prec):
        for r in range(1,k+1):
            mat[k,r-1] = -res[k-r]
        mat[k,k] = R(k+1)
    return mat.solve_right(res)

# routines for enumerating the restricted partitions
  def OPart(n,k):
    # Partitions of n whose odd parts have multiplicity <=1 and whose largest part is k
    if n < k:
        return
    elif k == 1:
        if n == 1:
            yield [1]
    elif k == 0:
        if n == 0:
            yield []
    elif k == 2: # optimization
        yield (n//2)*[k] + (n%2)*[1]
    else:
        for m in range(1,1 + (1 if (k%2 == 1) else n//k)):
            for j in range(k):
                for x in OPart(n-m*k,j):
                    yield m*[k] + x

def PPart(n):
    for j in range(1,n+1):
        for x in OPart(n,j):
            yield x

def myparts(n):
    for x in PPart(n):
        if len(x) > 0:
            yield Partition(x).conjugate()
\end{pycode}

\section{Table of Values}\label{sec:tableofvalues}
\label{begin20437s}
{   \centering
          \begin{longtable}{||l|L{5.250000in}||}
            \caption{Number of matrices which are squares} 
\label{tab:m}\\
\hline \hline
$n$ & $a(n)$ \\
 \hline 
  \endfirsthead
  \hline \hline
$n$ & $a(n)$ \\
 \hline 
  \endhead
  \hline \hline
  \multicolumn{2}{c}{Number of matrices that are squares} \\
  \endfoot
  \endlastfoot
\scriptsize
1&\seqsplit{%
2}
\\
2&\seqsplit{%
10}
\\
3&\seqsplit{%
260}
\\
4&\seqsplit{%
31096}
\\
5&\seqsplit{%
13711952}
\\
6&\seqsplit{%
28275659056}
\\
7&\seqsplit{%
224402782202048}
\\
8&\seqsplit{%
7293836994286696576}
\\
9&\seqsplit{%
952002419516769475035392}
\\
10&\seqsplit{%
497678654312172407869125822976}
\\
11&\seqsplit{%
1044660329769242614113093804053562368}
\\
12&\seqsplit{%
8745525723307044762290950664928498588583936}
\\
13&\seqsplit{%
293618744028817341095271816309320065715829741719552}
\\
14&\seqsplit{%
39383702222786673926973162381076522518507786667469626679296}
\\
15&\seqsplit{%
21150238597201101682069653858523961291120582792608950528438094020608}
\\
16&\seqsplit{%
4541066012746127523794114014353677637554989384469386170464309880710486%
6123776}
\\
17&\seqsplit{%
3901019453163597145804211405575323873685788280488780801307597598362378%
02714330925432832}
\\
18&\seqsplit{%
1340309287108592040631456418897079600732692262980135265243480395982367%
4593311775472092120554274816}
\\
19&\seqsplit{%
1842148103501576330259341331069776540862527169944272971274211116785218%
870704083461129503157001537715368624128}
\\
20&\seqsplit{%
1012715929581700585613681655233620159191511778259156891597878118890449%
964327025849578125695019384596395869369232406347776}
\\
21&\seqsplit{%
2226999137933663296854759620887600536728599120665505974525522236126730%
830375963094754875023816493012915661680538919480601183075172352}
\\
22&\seqsplit{%
1958882819870327857102429940141626435850026441141306774806940789976893%
2799351688152780693582858732032044214430634232341070338629544654409476%
079616}
\\
23&\seqsplit{%
6892215250354967563613516474489356348087240624399856716801311139461516%
4899621494982491897134953858815239771407402330600268481059000947825908%
7817615158466838528}
\\
24&\seqsplit{%
9699922563845747069044308682740474078733701159511468851773748229722845%
1981545090738063161699981220240205993932587704016447273787657646768245%
731249979622904089292863529025536}
\\
25&\seqsplit{%
5460572138288918776620468992538302632686412997582183752083816835747605%
8859290345594033710189777492653780301555646503875182087202771205827378%
610916433830771962609525703076373497856699400192}
\\
26&\seqsplit{%
1229611362324567385181693435444647111797633998042488153373322592347235%
8757823453395627069735987593284118973717878346456068996230821395538452%
7271383078443664719573242249793005840341908430393269940016644096}
\\
27&\seqsplit{%
1107535505069171020205551830214107300690995343511160001951072909548376%
3919639027812791931735784256700611870905436852789894897437593050126864%
6808350747654340639184589725113199321790313113546973161928185414575003%
4020302848}
\\
28&\seqsplit{%
3990317094336317030176581145332049442087970163809104474134807532636312%
8663018768040270694388452870772978887021408553895827567563637007320792%
9245103463594439815677681854732635334887430527595248585202852348447128%
95177470944146175165464576}
\\
29&\seqsplit{%
5750653007102810983547996986369455150362150697832248315460169454239279%
9880987624455950715648224777672349889601895555806235706891870183371327%
8521764820903271902075161790471056722809410709759428576064878482718537%
4575431043427940796170433584505195407081472}
\\
30&\seqsplit{%
3315025748456094057792643157269498180309170352367495651946898999048926%
6393254573288359165678046945621612592891497644930611840703319691246287%
1989397626926142439475238840940875440382615982664011047796215605508428%
1381620786366808071004299571412795511487391973836053341536256}
\\
31&\seqsplit{%
7643928942071022652578470171956531864343400797193588809119377488977118%
6792541641921090196821033690917449732880563786606996632881097600574589%
8352790494856376340734539199780801259837144779245892297256844944872605%
4682807492542078364196269325117130912358731372786165189583186680364295%
077756928}
\\
32&\seqsplit{%
7050280045376519365612273575246168932684752337929102256034766652779445%
7912869408560491827833395439477088900863363116423850048002004602030652%
6854157509230246825672574690336255481489493698602967432285727673450592%
4796639761453134055380876054431252248726326332551809279950425991119329%
7214705360019919598870593536}
\\
33&\seqsplit{%
2601094230711219664637248216759861459535198024830447058981904244294738%
5366415824023469261730261301411306774658146956672895571694027654261875%
2486267859808760651625476065965989962871158140160419160074037224472076%
1992661441554311117028001952087679245249158765354496650556528066802081%
556424181241146795769991650918148429463769055232}
\\
34&\seqsplit{%
3838537568491334060351260129642226795231093847561862324469690616566624%
1633374609010636727776607752006912222032736533868355400223858404042503%
4535766058498705506287874123465441854064485469856527634848182945126829%
5114965237222697796371360692191617021599879065662146203323628153426088%
95051724556916617080012240962674735484818364564906623227448122146816}
\\
35&\seqsplit{%
2265872643827308584893167121733730540026787270876211128940030742335031%
2205582963594824150369228646507980813407675179160358401039781603240418%
7018196653062566651656788521804473106057708761800543535820205521366317%
5504349232774388530152494239350806152037177276862912462253243540235245%
0849433833488299884919640114352564063620893567558230810958366739966903%
9793585189927518208}
\\
36&\seqsplit{%
5350140514784325936020510988865787132510303639745401439928638578092422%
9301633676082636929013084085012297012851984656785426362086102731088112%
0901559106635100101088418004669403093862103736867846635347567672375454%
1938395238439529954396619259174241204166562461290612310714692623173513%
8342421488975799578673590065340411875503441876314263893781914880518478%
8333407340981078743807991349143243063296}
\\
37&\seqsplit{%
5053064848503770909520541365286431776361456318098048799634319094098438%
8519039964670092252296743609973846675216766546090706751318828180579806%
2407117187475165827196129595609539254771248795731472535793697118508899%
0286985312827102240103407833122913197211412288528463317114783811620393%
9507615050968811197660306482289228733098846516199854368743341215933104%
01569976361987429854418300710295787925786216016777093717164032}
\\
38&\seqsplit{%
1908993926219090916722977584496712514977938855034953930707801384983325%
8819912673921383794678021731615790510053149230626027422172633111143393%
3104928123086843575507590622078703458523258400368699822151814061364275%
9378731643573901983239956119538413333844789917653361337598289691282138%
6724012366619520628550310284412310752883412734481241022096374653935910%
8932168253210565955867203678801299123371399131147414143200028709501911%
027019065327616}
\\
39&\seqsplit{%
2884790058508332813864705537551608834999095493993210545083391279494255%
6489676124230712544398992968438750986522813484779487947003321604837933%
3842615146934578559069911149793211568677127980066109034672604500749803%
5959778240836977133696872540703162381993756275019451471374690892284550%
2693993272681951079858284266592887244514831265907270149093840310573129%
2619540808561328513737884850498624037311694320958796272825240665017684%
31965031106103112243999592766093918208}
\\
40&\seqsplit{%
1743748592980377591776078068687714603386087732627138958375374491530697%
8736321487706955483621254420836351722145312068677309709008057747455772%
1382518774849558815714937838194335236548814512510108241950897974378685%
4761005023471324082322521955660338684275691406609135425932678631002958%
6008058873657812628814012116924725186345553840471787685908129166179546%
2997998581792242287459087414394219208339773097205152582841390621966537%
04464455495056473242107649807579196449066739750283725112344576}
\\
41&\seqsplit{%
4216125393891032502604383022471048720908846352082796164066226152869607%
0238186593051081060140916033941018249943989638244868637324569081251439%
2820174361236731312543334307073846693378798108148247041572726111345978%
9893228474877365376790406020410929087540317243068884050460327884533861%
4771258715359709681138734936994037959238583792353030954849620055868172%
0318654660010909493304250149643395951373270255421392339492181739934608%
3083876978768504831551721196572227091415681250930134687713105844828006%
8182815751012352}
\\
42&\seqsplit{%
4077586277949432422409655337038633010825975120869311029557823725085512%
4831959008793886100767949899514360985194002643571768026346990999701425%
5150371815549721202806119860754791061702894852372005871224703749705155%
8693624517472117187033426599932984527127890576435634951621969712901662%
7455421761557818421369347398851063632593787431162103629631565084776848%
3420451731370385903589690620532535074439285010050507327763335485265537%
4428390914178718248540707972387453752435148288338968878070273886142828%
71072341267379121357347891421825221001216}
\\
43&\seqsplit{%
1577439786592893740848542780322394461142355413461680877686702813908593%
6702374994275651720199431788793768701716433152185060926712989840106541%
3835015595551681384319701435416929805488718756600383230843372350535370%
2428473030228456309570975392774950884615245900679225862021617218447450%
3415112260317037401660647013472623338191884238979842281236236543356872%
4716132697477695282818144658663007138875818928700886038787772834644959%
4816511861121525173169394431012315295396439958530767465880788543232690%
7254403569876310992188077911800441476540003296822995963813904252928}
\\
44&\seqsplit{%
2440969839235405019006439710534680466396096677500342174481207587485360%
9834700253111354529650734680839859530697867726403502413877808687763286%
8650682004671857373127814999395701547147470067340642370583963962243160%
7021397145315923042143811116768895737267091591195503760120090031266407%
5042222049884566993399143823270163347103336819986728971247909597083156%
8791647970557474777659373226178868395863959454706033967750952741432676%
2857676621326579329572792911916720752128099832589166870091586251805948%
8207828970950372481169360997181699216141758959429178329779368998350935%
98930185966899488096256}
\\
45&\seqsplit{%
1510887149337364554100522156795204756852380429392088874693592649860299%
8533450909514161400306018440761626965498702000324385206103641956868421%
8940902276470506546056534302905617326907437845305671325519119450153196%
1023233493882044317095314988003817760569753856528986023055016119595289%
5294765747247530566858272471365850088038729532091791089597440408746166%
1345855439539837131884660816547128547839593410829409262023118711862131%
9396412154242910080370135580528486864502056654459252851307854901034766%
0101218980641773997535314943382700055540123355810938633406199654272254%
69650263478673904957490099217404283923827703939072}
\\
46&\seqsplit{%
3740775394014725453207990234144190903375768238705560011795668695143822%
4444812840200194125659064681609911755820723465488042357402351263879794%
8782972278297218072924916483836407896348154128788608474300919248088566%
1975911012532726601558118988047979320054827383950237461966699600894778%
5535385782160045476496104373567161477532790225455852968320658113984280%
9886527653402183164437947365774551866371853913505528962935153730913498%
7366671949789968240471264272136085769775564144452674671722147440027079%
5558641710997267300705224269822804035151007095907135786038783786729038%
8199873596973851473884177204393997354586905131485275380695162831794286%
9917696}
\\
47&\seqsplit{%
3704684510578505090117066529481152524471003462009616796792159221125475%
0307497929396600004261426056537350899564503464990316025491321938909084%
5604482738559486373092184507017120330526209394305999257257888185908583%
8507017771537175274754434457399530775286198124879856059952282070718506%
6201527517606064429660222277806456108856578129452461158908691968151072%
7280155192111931234796439461678039734090796292350217208392930897875900%
7527735941295074865670314585327057394721697337281305008553798992562484%
1470020435156504004045959997155697773215994058278567984409030070076143%
3051784331557617537265673374388833384633141048046442971777729334679011%
67903444081130150262624746647584768}
\\
48&\seqsplit{%
1467576732341144708339372675648522394043435656993663253025768559200964%
8925475356828656991975741692379720979326485898318355129591802800221461%
0326420415553960792308931813488895439541061591975343614040615921604235%
0613051958705970160813598180877726477375268621110613471063507415432083%
6541849048817209025193100799761454018821791386250974431620946309107428%
8276914432136709033791357299342132355241691124546824318222768930208646%
5016835455224313869839015619793773318035260239529949608883691919309247%
4216768668897093736638715385506275388467394941180118926059914664590577%
6961757211940616040986808906459449715783726415647810879087425397134895%
1533717861045355409089854242367344169805174031553430471894368256}
\\
49&\seqsplit{%
2325468157041375289571741753754833939865250528097400002116095721905752%
8553749109625643391720593896203677560662143569074450668752317179368863%
0715483836343808927991726680567770503142501212928321976610897501452105%
5805257175521165947503014903918996389021601108608588681667563450697540%
0507386597122722675521168497669744633586243320275405749828814063639145%
0342889704144878211433852564215250809804221209136413218959530026727826%
8296756713607532597692405714445946537132642789262365623362189521354754%
5503633392233338855731423274530937002902175385380994572360270701944845%
5549445931090278056634362019156100343449523367976227917475046410140697%
3018833805162461578510836340387125648810437983754752253452781922276245%
04685159507133886103552}
\\
50&\seqsplit{%
1473940552542617038638273880643841897199488117373648460498368369170426%
9756890580475784824108118285313072708850589145369195641791169216374139%
5898196837016865151660821306184590343819210292951343194120736974043603%
8474878960631065709538540144445072907308675256486743932998429143025252%
1696199534234373217440316818446428714647967499867106245101565801812889%
0387231323137544059956722085215635634465593931797128114209037710026872%
0198720533311344144021911531645530027740726541526675875498704782324971%
5607576627768423602978535802729196377424849924745972501997212205474345%
8863315082922703068961372560006779115134372803124981771081872889247993%
0695035376343007308818225787822629385661891950823592099366544366847614%
88568290353609948898748859794730072019772042029039616}
\\
51&\seqsplit{%
3736883252262683527165624376181671379067373040643863329101629815752793%
1487438447100047665914604008411435672686256023625189630677317662404135%
5066758236803243331138067292174784148282610104331475297032132283680471%
5526446933812348796925857864061923861555238245903140319247700504842948%
6570538764407097432134866365367769321117400456143998669308698296189390%
4942724002529650671352666016090902285434640040285505323035927234882161%
8624339435532869613335393841755863119573126947904328633231046406637556%
4352162602847296735185437614222396773600372761954801032939924165195295%
3706586903958310149982334608436174808926849136685358755068520322055928%
5965896385764313910323356272359532762009592023321488741654446066462688%
9969068504203218095505180782087160868251249541616529902930997382394090%
0604470099968}
\\
52&\seqsplit{%
3789649838170920215266495495786307834281286058451592060609161651898632%
3296716801250727496668200895639275281341366712683160584906483334074584%
4461354806489723868452744067142707260839510255606013243828797611556773%
4873239887750945650926417035263124174678600232503256242499553645514957%
5471443918989769716338330265871948290951035182555117306913058634847449%
4550707310986793000831005554706139227103843018545271227559288442419349%
1507599083658130819755098179521009125325468810395220157750076134712964%
2618402328070014841560417972118999229042179681121279496706928996645014%
2567069671734478793698003566818348606023147028216822149020808870619279%
1955644230700776573385830637412581675708252203047426384585671565726843%
6910666259144199487778163240761196421460360613788899486454617275385769%
49211888240642764936160565360271919432073216}
\\
53&\seqsplit{%
1537264605443890330426467441041043915280802782229483282399023879221292%
8687458788837712195038098400598300370507871049669344557287029860691816%
7526538612369628789221172406072071562896299784660497541105663838719689%
4317527219901025131576680714946168131922391201328035814117289549355731%
2630272282702858200073235529591722944479250433100634234803912250721782%
5052201688093885814034140373922612895375030647554177583682998437748849%
6848286538922779227927637339116634012410776784965361855408926409909555%
2159160687072155814799873351269798875466338689164270705804470162232910%
8903741819274169947526124755097960869558299110378304225034476326790146%
0552223876262413042709146538687402579686097008631377626986973186875270%
8923107162388743212666310675452890392373929808865120263539143252856296%
4826974717992043410045600471138148007786278015607074159270190691942884%
442112}
\\
54&\seqsplit{%
2494354431744722008927824820227677902424111865819473758681096270729423%
9017037477748945499504466156402884779576596641288453655973587575637402%
4113033330761081802881753245342130283831217995372929120982618714345216%
9291808961495175370603269963021885991741418390870736839255726135505228%
9007537937257887327685557203216629894434290194572388202673221621472545%
0009799953779668390342430872057182832406725672397220248530184087622966%
1411079812958584895992694162046214167830233957797413273799669808092460%
3644338247911118126171236148448845162887520145609985569431266589158546%
7188758041874094247525633450295459656561377760373604990265867966337226%
2995374316790878501966591810888382124016894911728419214843269379633297%
0261432551025929980820764873637856537997262642661514029866298018731270%
3556247360569407398186447300715492327567524115258029652949262051976880%
79746334543569659273614886484528070656}
\\
55&\seqsplit{%
1618928585002566308974672337319522685726218379095843182797389088590883%
8114545848225953587125763576489275991625630748194204687415625133978712%
2669947510035040600579804324023716988604328408941371196022121442201342%
1326133218523771435372362518479448687325857095185064585758640800245306%
8250346086239860260390536383186820579887870535671646330605536557317938%
2656142482859086879176493491802576101618279732336166618356126370460343%
9977106506147075998192686570989063946697569829408870767309599892494516%
6359699264492381559565910544335641860167595576666320569610988189092194%
0580466314580199066814291099298191048328595981293989894610718113688846%
9691155694173428996896320644098636168721056964460921074695442176344533%
1774706218645250407909687759092630771430161537879792893008394350999265%
6250898880271215376544727310395992752916592376533815357211301335639834%
4610171114365661961093355888054696462161682209508532568856179484970188%
8}
\\
56&\seqsplit{%
4202978903050531676317502968592305314077487083693429347352277250819500%
7584842135615530468743004929215415702553237673436683067569322774626542%
1199920402916865769206488438774434632061206999452621515614940812678402%
0517307657895503971321034448880909476116680045540302191994900942036551%
3093872659845007715957368303474973715017675146583153631698394398346952%
4249114861784986583605890630529200626601382210962151607662098442925319%
6761799857168259276425877471062609138089165304242975653349847414156330%
2830043927324122428955479939549759172977767680879413123715845436778364%
2393704981811864947834955784693434462499163153740144716450311830766853%
0654990137804222182236031443437758564595451486344257096791862606053951%
7457006061425389901893047385611479791806302694534102445000805973337547%
2728505049264797547454514494552782231704748498415143267273993637168106%
2316746070177255334160945765030828294242342696422986300435442588045523%
4760201086650241871983982564868096}
\\
57&\seqsplit{%
4364622830959485001760118412460641058392259158642714870796007271164151%
5281016644437290752651885675861025035121080429774877788301321204761420%
3118152393075961887476119067883650653467759566398876949418746479750653%
6727924718695375495785629729438862725230749909692216320697401709830431%
0855003052699113086527159034739719862157244351271490941989360301262803%
1384915611114851701053261753959299708683746291616419118262988026491782%
0334582517071103292411410227661523362257470143695550137327937508068806%
9564708344673868238894658750649670190685682517848265753228126754827627%
2585472719783981782048072697670031008728665167696054217225722071637086%
5848879033916649766372024277587754731814554869881038272276063551207982%
4642225936321062218139431370291149117328022597327540145165632228994194%
1485310845435917440020587006831873119359767729047817522304069238639533%
8466910873157724885061195832878030575777081365395005800174304921987075%
00824398384318743508763844417315346021665184467576058199231925583872}
\\
58&\seqsplit{%
1812993393110425842214698122234231772405199159137435055146333629637537%
2937763082684721015414150485639112368128063662347952408927347229636881%
9637750337025894241061336299285416984853019243731707038106342175072681%
9136059895573792178574685542621886385696150073318551735215727155460115%
2669987255632750457470213228080444116376232567559879453226772342939108%
3456890305139694748044893374387602278386263629373013150074311198215433%
8427197973284742915409987919522132387351980104212205888649915638538386%
0166704677680505427698805009112742112717260038134398176273652229679929%
2037638498753791839677173581881079070825303113513614794477478195082794%
4461460320157128950252353120868280575440182045120402970643756938755380%
8964685488707234642057987368988697165295770374037559303224443416784746%
6281262071455302616479156074981681288303363183982979239499914349636535%
0529556214797038575929898688410764297271160509513675978742863798795810%
1277524647719560116718123210281887663070304507656679538609106589317525%
279128313571041632604869892243456}
\\
59&\seqsplit{%
3012351967869331589452401702330200498889101082435075382841332615406708%
7478614980087492476339362628935501856264982214191457793355667588588036%
1654161063235647309480617178813843461388783301553191304496863737596783%
4371814622998524365390371709127685611545198403988381628732367451825454%
2976138540449299771094958487035250212660024142931338098488971704507841%
0944204860176616989042308757215675369933680681474596686140994694662886%
6098750726096728440630019834016535249160883893494791810501953137511541%
7630649370317989159575017376185282659697687166444677897592722119480561%
5827818220816290587603675573219850598148148802484468656292462157107382%
1437678001586326063265393070550645327653163888816716588463662954335962%
4673360683778116473221374200841147859859797796519877184047432964950065%
2538063845642212925409891463887575405705753911462835427297086888841959%
2247092014097828467459459279787940124719106800594312231294186138466167%
6275740461551041343371761170053384410704789059782885386696315673024205%
67946006631053316493856390601486162843007873050369676644734728142848}
\\
60&\seqsplit{%
2002051284424849528658218615896102804343911828798299643640523617822692%
3202621651818158803200357581859043980227523940129431884715054502166751%
6108537377559027647238654919172320966834628735375696435293950292161422%
4973150275904857114639217397540575490774362029400172830755485116635702%
1566639294437851647941991039669139135704979346133502422576772550717077%
9605809754827013050316627380595192969296745763885436461290663405515953%
0259733294591961092611579383200353225666324950059195639359408050467467%
6521248665294667572501056261470980089886357149629200152071802008488363%
4823664861236385038299510089639420338840346126474208898831801051609326%
1316490901636710732142068560611334861678539897707737864592120622858736%
2777128369011871496771339263199415397950958594311582111805478342917392%
1829268890933454321758084628996857905274423191590322596980512291993968%
8792211378682434882554809677354689680296376578526395922768330928667529%
6879045127685275445423744116212132442836749685987765864462075414533658%
7657758840199756993552085260357497917127290366030841122160845052151894%
6359744770894275848107123619659776}
\\
\hline \hline
\end{longtable}

        }\label{end20437s}

{   \centering
  \label{begin20498s}
\begin{longtable}{||l|L{5.250000in}||}
\caption{Number of invertible matrices which are squares} 
  \label{tab:gl}\\
  \hline \hline
$n$ & $b(n)$ \\
 \hline 
  \endfirsthead
  \hline \hline
$n$ & $b(n)$ \\
 \hline 
  \endhead
  \hline \hline
  \multicolumn{2}{c}{Number of invertible matrices that are squares} \\
  \endfoot
  \endlastfoot
1&\seqsplit{%
1}
\\
2&\seqsplit{%
3}
\\
3&\seqsplit{%
126}
\\
4&\seqsplit{%
11340}
\\
5&\seqsplit{%
5940840}
\\
6&\seqsplit{%
12076523928}
\\
7&\seqsplit{%
95052257647200}
\\
8&\seqsplit{%
3153668941285723200}
\\
9&\seqsplit{%
406198470650573931200640}
\\
10&\seqsplit{%
215366179177149634500004545792}
\\
11&\seqsplit{%
447870507819487666185959047316144640}
\\
12&\seqsplit{%
3770394197251690930118967532374966498493440}
\\
13&\seqsplit{%
126205342254129164806148123600990735262978861434880}
\\
14&\seqsplit{%
16960349752279776751561660450391351891796348875427924676608}
\\
15&\seqsplit{%
9099421507577175992020382974051470776190067413539210813509193441280}
\\
16&\seqsplit{%
1954628853639228355030574032786595919541842428407509777250056186030433%
3455360}
\\
17&\seqsplit{%
1678712601918494322160545909594022660932492540796435169158901524539401%
98413194412195840}
\\
18&\seqsplit{%
5768516087464993968410424003693368746051677541403587273429368264074190%
563974746362542357433352192}
\\
19&\seqsplit{%
7927815728023521061284828060734799773685486908763424351887850385348492%
15731409861941064155848716971103354880}
\\
20&\seqsplit{%
4358441796213949528280283318775904914878476215263889071132907528084057%
19839273234051108126200323648090580227063866982400}
\\
21&\seqsplit{%
9584230891108355033689496585456324375236727469083890578514280468017158%
70107112064975179622555650825480533056258670356671991237836800}
\\
22&\seqsplit{%
8430410870685601645830764475814222686627609653511220501860019402401487%
3147614211019255968334586197970987502330989051799449348720976133170428%
64128}
\\
23&\seqsplit{%
2966181215042382920178973222982059020914540568145327588723540102487829%
2265897888524948467295669699634781456169755011079456694590444094122612%
6633036404542668800}
\\
24&\seqsplit{%
4174531025157668548914442907453932540694364331567257672771939204273675%
6129314900748480872156058183426995512636315501224225218535451829746865%
476526185532025312537189411717120}
\\
25&\seqsplit{%
2350051430802267995436236790390354850401529913240860764141274367099280%
3634023305107756244730458243055803909468593647529576389339752372461259%
258325629682822882877248453550684375328269271040}
\\
26&\seqsplit{%
5291845855829269442324846824211698671185831331871465134605069595567015%
7325126471939028717329703339117921640419065145791252263714377980812693%
310011201334134110228054691703048469617143834039649837789478912}
\\
27&\seqsplit{%
4766470860709279561616903575060635552310695214456424204234830229733635%
5226168457964654858549028217900717138898730310518703433739826235506430%
3218819137992655931419391480613728565503408095285972375374755261178263%
204003840}
\\
28&\seqsplit{%
1717302096951472857728688027887992401411103524533732433531976480656492%
0014077649364793711155359996942891297779049163815304706809833128259920%
3132270971419325925807910377015733721044990591602340340294499225227008%
64323766165076299271372800}
\\
29&\seqsplit{%
2474893158500339996288107226767709069433472046654938468532183843401638%
6869219318356141848297978719844137386263649556404948690471441639635102%
3233641817863480282386594557103733387299687350052232449741820141699928%
9251389178449189887719672310402844154920960}
\\
30&\seqsplit{%
1426678760471499119402249736723250148785682303413347584234127160645500%
0995444075354050854304663898562913848723574175181631901810900700373559%
0234966290531067553872537859060054808513372049017137130777079371292135%
7698494771858292984876209754057007391814879322793193188098048}
\\
31&\seqsplit{%
3289697262955869298499447185989976140212381405380479555215849055162305%
4332320166301497529693458562510905831629318691271663679388036985295122%
9319021150527066996012764659167824862320878491839430662690457532805085%
1654068449297955725628593448426166601875632258606293678253413122922069%
801041920}
\\
32&\seqsplit{%
3034210160252263588232141235371128717793154576876828912457611080124032%
3049212313030478406545507620043444989683591166659943220408118237543380%
3724360296110485118810913825764702916018406919855845878424440777662251%
5645241963475361406672034945137133064165111145527337598042522513710347%
7376611185184437850477690880}
\\
33&\seqsplit{%
1119425968721360105270446033221497173318434598626958537376096452435297%
3930109784098213968568525783141210174434640008896718178621105615890143%
7659822916873624085015354331394264291705969328525956499063276402928620%
6430284517248567175334718143893537132254060174886688218467099532413693%
671735224850101896593082186315783855071391907840}
\\
34&\seqsplit{%
1651981145840914769648670637552979884184712350758098598200288324901807%
0091584905028348225560351646252095713425508596033557517539593911997379%
4417878240198576940552414887616090824225876829256904102181523550647594%
6555415073677649811178155130891711188687672569079547477207671485177180%
81583540979919908303737778487650956931269569175756596523172083269632}
\\
35&\seqsplit{%
9751575499899641723095548902342938922420557219805346483159841645647488%
1440415524263846585290554697423869138590564154149370710752083219752685%
3496936634095124178451274392151907787209098750516056597457205688457671%
7262543713358054474605149836401416820165889614067399109867600494633734%
2754464207029187617229321149786590603361374647562098500183890276083760%
008682250572922880}
\\
36&\seqsplit{%
2302525663633796537126645512646129929514167528478094073274183708568958%
1777502259794606211876609782884419248820289513979289958576308055097263%
4720805363114226629714482723735979605003019907531663646604601133568190%
9840361559565464331978453234822751866504549840062599171370355638339176%
2698558948310187756991432661673594752755812323226964459497010666433839%
5253135163876964531358774442536520908800}
\\
37&\seqsplit{%
2174674004526455294747852629724623596534766529640779471088687184675721%
2992057525146036122821826319376249546198858012688016942094992021559832%
5059312452391329779091726138458893042106204907833201852468640007563303%
2485526448794056744282366310466624409456532710730341633302536264206745%
9002378899711055449544040916638002619346623931953026500704352750923214%
35856387556652293198540238497777007966161121016580698716241920}
\\
38&\seqsplit{%
8215686102910429044261197191399221403682933882761543958931315529615936%
5504654181985252765043936481423747144548628717679645455379685539486895%
1483191634168961118897023165570466760988034359890583680872815286765181%
0365832932410292702970553220149030457956893889795987713717319612229830%
8002729316653792895819676323160012892511127680235423544377075893515817%
0590843899831157257144485767571030285386090750237367888800320943767357%
96186818019328}
\\
39&\seqsplit{%
1241519382052440404634180079698427169988164424982066285480701699655222%
6231970225034736660392230563365136039818523931475961662838248308619105%
2047025152488775580538544149076506455363721403647425385761462148522909%
1203997335795743337410914355084709452026071504157159420538621036235917%
1294956850425928044230492029558346205446504582206240773218844547982479%
1702782916503973971061343736158990781177105611580067141872650332902532%
00421300236469019510828743790522531840}
\\
40&\seqsplit{%
7504524182237522321440305464929611754688660222214967711374900192802224%
2181608595969617412133269744211240265195238465490715792800095729245201%
5143784992686929799551186028921067643075428947632221646454456211944478%
9426121927103473227779164541928359738981458249970775177245649879331002%
9031862380055554216747962923829990415195800785361158614206962979174307%
1837218326158878906057627805582523897525116392320588301469214380891233%
3776322474352941794654424357980879870646253820147176728166400}
\\
41&\seqsplit{%
1814482609606787535793852258473337802891922272240870411497698272498370%
8977416192339048653720763964131953688352451012530255625870589054818460%
3255198783941418595193894581396414177682279588129421749709313402509342%
1521672647469151379778475089487120855560851167256476320221926166336812%
6178340352764130224610643847118869866667615180419357620904240539716270%
0421287054142426756097049641514807570952459709008099835792755453390103%
9388791240770210250818030190423990450578707940661527031483026371844792%
8566093552025600}
\\
42&\seqsplit{%
1754859900774827301314835053502975654906711675227737781939167418758698%
7252867987535872831513889162516900529732520742803489662089674926208286%
2000275803847972287757192240918853186806362329687830149228541345649748%
1351093799497827737901670857127099763815072460670369434025663295902519%
6828186219909401093991331448374858022625636586233764551266189216449645%
3462738642571335448328419983592418929372415320128306592049153917121208%
5938553779084222450992022662108694015878733548599343788373673511260930%
76581969882047233236438251810727418396672}
\\
43&\seqsplit{%
6788785420372040569663059908507422301505131555491501372413752697687684%
7444448791960278989509699128210326036920378223802108697956363502996849%
3029593629735159220282874454314456601385845321314835643582848699624584%
2639898327889670040058365687774236068550094345482767713469487116120596%
3333198011710986258562555249801439993830340100289531257089672758749879%
9407756760576706810720360986363274735867485812710296953767048685575152%
6721114682951619377561240437724520915049997076170248903926365947940023%
354828926684239315809951487729545810066923200601819006035086540800}
\\
44&\seqsplit{%
1050513661245833892310077390339868962071512092977739782801006323736824%
8709688935763986139977376916677420837397179551047716681571770436172003%
8896616105780067766580426541535876462986918549310392362787861689424302%
9835048829759019482982420709986844938234814459254663938654652900097230%
8803661384157971899469620351091222976602219196375550547660832236690543%
4490752975797284688323937802291186242985371394789888686958159408699421%
7626900678510364558645354003399232222519673155591733938672257206691154%
7522284294994695054959822426731481150205827247651459298097338848042207%
76689219729262714552320}
\\
45&\seqsplit{%
6502364615374032578933774327351190763016878300450482272114573238000491%
6712582830300673910615442843237086770642094209802258127155351499189043%
9806263233739458015329785200007572137588697308565964431243345336749636%
6534582911714286595676919833881951343614813939877618291710794250440031%
1383181676158669231968727848476240541226126763767335294978199602111036%
1672229156441035765590180492979168385803149762388850485634343030136936%
8849208226397620780018684349937595127558569687333770273823854041667469%
3771516370410859740033144915948327521454000642913667966663998296016394%
5696727004476904468946510180896787072550541721600}
\\
46&\seqsplit{%
1609907501479272830463863159087089023572107604463249450946096019576029%
4309592289137999777658171948078549931656899231633222167671447060871292%
6380064493257824280372989849721195800021709008309331519923164931746864%
7944942260068422571070840027044462812622815046542047523638657429120741%
2986755335237249221393105108266715876897817758369043461232068761082792%
9332541202601574690154168738048304789132871186363612339371277562922936%
4775499628735969287518629421212817978438585495276378751621745106139288%
4436919534491566701732189600918313480999253793117838015525440795819463%
8886265086209007373211460610393546534162027651182285435279993013000899%
1858688}
\\
47&\seqsplit{%
1594375164502394385774339867205914211856530337922439410897593713860368%
9081494633547249728329373167674823184201563815119808562215838332537872%
4797002122689189347647530595307472676279264828503002280663441683551169%
7651732425291036606870459201427998482054959077862364923399411638270813%
5188555004098493049810196252885203805006511434615368538952877768728084%
3810702116429038026484073653043827535843782594443685787407005051268848%
7826847771174807400781799341848001556940543436275069581452481539661752%
5993292686450591378316433222043810045998215079683328537928791778197735%
7839641727515077106286717281061694861503039757948551567152900236557466%
79689670046753761884413863982530560}
\\
48&\seqsplit{%
6315970732093582400356931861288836372639658751993263582950548170183311%
3354582368329510591701930093695624849230664680123671429129240484908600%
0791888080019999793684510776088850979486587413694669511525675921013972%
3085832484530536781284595366813980317561662252003544637285445644085321%
5357385002286236729988608646161644153996594102460689722284263221255279%
7973868928659589673332714978478456720173049729328897224379147065784200%
3597890545211170188170696348435534528342195266337424331415576847096758%
0583108260561236556788615823171539558465034022280498117031723793594260%
8399744100323044593530946528464773173776357918624820062149300384385622%
795605339782631770132086699470151542088945652508021842572738560}
\\
49&\seqsplit{%
1000805511195413921527504786944392314975907872147139956071212396465035%
6670865972648370262683470592525422300395615296146990819699168560162500%
6810494347228265428961897577723954980557826473035476515646534132058337%
3122159574991050058941463622606112350942769997690726351305620423735589%
6010612970426588498887814478860158848765448853478576970911611657829362%
6682143537214460180018338196841259153757000181398778931470705711648061%
9736795024610175561026601079703143553864489768986151927624187800821072%
6857294332069215540140399151744117165498186817503319795032632148327497%
8320849412662293177510195400207081977156928753944645045655166244708104%
9507707878677644262708778189950655780012901588318120511849110115576337%
08894011211321127731200}
\\
50&\seqsplit{%
6343358534892543381097580210476088736395751118941386590518411221643870%
5160829266709024216391457714553794019136968109796989561575154158698795%
1180029975482255647134700305142397837284898327084130101131192908813932%
1612894353632862030935067010472125629744737166258064297209146168366996%
8578964448122604592117072957433498652659429160360956237790030395613164%
5587502232267404696273870435258381918746702316812998801459327921665512%
7279131704915589653665165457363724840432234976562301595041045166361468%
9216064017501662498992220894496359382736641235653535347212643306040162%
0628894573737606118917536239857690643661465520889720795032440996740045%
3919167097554821483582882705191458388219380453331357254304155220478185%
1938208118527315502083056957618171534461535824904192}
\\
51&\seqsplit{%
1608232450843926382985303968395501099287130075727167172554419192895457%
6590968117788524918854201232884205369242277159584970847110221400514125%
5268980057611509365536844979853940158280100351095977429955128518010357%
1466538856605076092034231252684263039674881227016399725680569032253148%
9539663102809405663807595058095748485412796263869477083115120942489896%
4736084706900864305393805664245274500573127450625164884621441740546146%
7891530170856797426875025316655988646092513258624764662351375872021337%
9656219452475450024119010886817732243766322682500466940361816930047721%
9204695654998608224553441815800126286697314326310147767337872188641235%
3201140093348195929097870228071505595848926743299589064207865630330087%
0157190183495217477079490941867981367969691585550486123888056207092938%
0393091072000}
\\
52&\seqsplit{%
1630941465295031168968428674021791618767438435218025119656901880285990%
2691792389154434577659306282933036454718294907867642763060667235469540%
2077481219884966097949088948432824906875075012851536270104115970107848%
1172138217016370892649270810006128695821218356175192369190268195927552%
9464690979133094437072567611887736897276753638330383459179971539074893%
8193541033463993981636254300806903796814232592348021531993381801965138%
0506040108912246120939918135916916947490055872729823642055742267153119%
2386602121940954010020404853147192943173031567760986139476493285446150%
6302767439854516270401732841948025442619456486354084027474282327160728%
8456810502197937456872475202717383409633594167124872379188302614687927%
6911867350261910954792109614562935834609578184764450312123962161294427%
28145308814074774364242407920119758307983360}
\\
53&\seqsplit{%
6615884567738781981054103135658284199000732606367486202476550889465814%
2211272146615587188638330721398460912924615772053303008154522429911920%
9925974260279656270523848810027470768579550017647241921714794779419686%
8052076283862024450754883805615926136114758004149639577389138367545012%
2333970588937888152101554352701828254519776228104531977306829384775397%
4934702858617834002309979946436904262453002406893016739556024263227569%
5200282773785102778455246467049986314536002257711286177547745238448620%
8892723004370578456563541301016844347268894790642100355283693470427159%
5029102636446821417110683465293036813566368635056829542704919117278868%
4629824853453346884334511225908977123938570808066987047010395832218143%
6936568609450722614527599497703941145596381012290363377561594370969258%
1888799564497303940703775537205689774537528892537242642988330921809149%
95200}
\\
54&\seqsplit{%
1073488645546843549386192319455915206274482315875304772330288511973417%
9993421858724656210876889422857452144488283214572567140438600824284071%
2901939711875861776852103169608269086607249778154401854734907992570914%
0015940695784124512931807242869706828601171088421603057690188788653613%
3865339469681010714351583033226766631786313224576826102639940145673362%
6332864730909404174694637414001380577197478128348686278466216740940813%
5157485489509027537345656340220412010465480861481637198674373470236290%
6273659022995692275486735318685319616527833890683098382045418252167770%
8773156408835152960682872521811714322275746921287375748888239161680237%
3092121667286844947506124433114603254318293277653416049916513791854108%
7231576177116527800956816563363027330363230987594394022109488981203612%
8254410813015446233590963695588175887194558867137162527547761083047257%
40696178236584135813707763397299273728}
\\
55&\seqsplit{%
6967339652432116536045621449250887475143835124264234673020051731671962%
6400613206321280091261873574805528558651907817734195270995571820706932%
9686028566834645948907264494973591827006460658855079394298478350336164%
0380546880814715551267855333748519865275412167093339946686226415594045%
6000653894992067382098379062048229238712820814431567001744542025541845%
8254259795583190182407305429879908416392862801410673761893164495764785%
1836540108655910539587412081371653441324341263526815472063786038838045%
6010286407436720429766724182462380438412269308891570220222755266070619%
2812660375261419029436119579703309811030876713077641713905725863075816%
1331088901992434657702757239237821912915612176966435555856318397143869%
1946613103622605549752270227075008448719197151971824192811577132238345%
6789123524310171888847800862545165994659421298960617118889002466261759%
2857327615576762394846528542974256565319711932918827165829270108897280}
\\
56&\seqsplit{%
1808824789483419098220718614084563909635733521119176206219630878126587%
0258168293236260722171828000897060887274632134992140325647973977659331%
2022422236654196439980264927157459377797180940748837475045445426714102%
9780678123164524439304284834590544604946927920288406940027262293421656%
3933005004290878956090060418786684041488513209185946207204259028500907%
2282362636420513691905477526392140085205443717553465179983125941891130%
9736426753154504297440159235237528403010827541127292904281202133592511%
8630469598253545229502611091427892213742361496450435010542070108714114%
4337033314155158901137369399439960284935070311067953932753165973641818%
3943603928819154263581009157995737273018877338208839775893313365070333%
7831150150187792132455605328878885261247647958233878413426753740662698%
8084391752125496561219360870629013161293633339316478447718066076904803%
9807851898281188606341331679977260714744933325410737498170552467838442%
4224945172843967663330735203287040}
\\
57&\seqsplit{%
1878391054414936409238293328326772596304744557431662023329523620730810%
0937830120103200133240390648417769351744285504088468913585244391684782%
4737917161272950622213660556686274573913689898555824922589428200176315%
5138738855837838997107127929848810634789387948167439249236670083220825%
4934352795334977256754287366219248134235485309402107456985717753613088%
3037547075473586219270009761804079239782603087512542944126778614374726%
6364470139319396426269466952623861007722837567428138133324882018583204%
6984787900313953491099090482518300052241296537379530644845751524042052%
5079525195418913886012690809564213601312042089121723210243930261941423%
6929495838036403349667925588819490677189495007316505464955253608496414%
0979690405984985149751659521430760100582073007911412048370401084900352%
5378297433674818250603558801227575871753077221173792322009500754299611%
5812721189271870699350615096294193928580429009097110910820164408941865%
84598014125800096455536445001873059874212663781268569798038226206720}
\\
58&\seqsplit{%
7802531176750875664407086933242519649486999822857663607365961233566009%
1854672091317202592989041406335732154625179228852557197495637603999373%
2212380216878716856127685561386345087849198922951086215011652848780268%
5462919460277099024829558906682339701124128324824895554854311660964669%
0028872338656052276816176135830922538502105648445034949248053076633511%
2747866811415834851697778682591243685934809731248120065768446761080881%
0212864221971836954603524815951921618443475300290393233906764117846712%
9968645088849254752106494981486664674466302703795585076784243771658342%
2179692500918584942318940682167149052104900302838135837399198552320466%
6032035329545241382059899502731973360692497462365293993660790258047413%
7213366409464639101300897917249432217872517923250607121238835743753379%
2695012865447510359805358041154627389758550483167350132743996852755848%
5903177753163712160499122300063903997644671382496693254166323693201108%
6263108331668764615352044886828416890649104901640078403102904832640996%
77477483770726135660147168509952}
\\
59&\seqsplit{%
1296417859765236090650049477143720733026536839319195507243728308158042%
5926813427452134008559437794302512278373670876515830566403019092209244%
2296765705898124208264580325449393691990177561974666761693953081213393%
4583140529570044054385513939679698519998783110793602828174764509088243%
0767026354705578756472038280855082071949967459610162749217297819988672%
3376471283803552224373143427098096022506314831275227070280274556440802%
6471476130981940801474702082436737307215800808045775961610020065276429%
8073764265708573026609379668908944982880630360832805407651326783924521%
8928249708073763736278582171459943322550789323557495397293543098079040%
2141561802442626213034032894202929703073707378401507784998799609648446%
2544564032777829729153114824610434205142175215648777808010846589837848%
5623441933238054329391819289808421700241556253791055070263856879497825%
1815541535661923877415348085248900756672033696842587037281544525374454%
3980038957900570520797529850768403256582073620524515297597093623499466%
32913920935480626905259892165470109537035111461765795673633015726080}
\\
60&\seqsplit{%
8616174567177574039916945427756785548518269865671798112997861552869740%
4978250375645184816443796738665181944417395415647861271553377375172573%
2040293815480301513797105642348651413962558449799674570245185304284055%
5113457408230064486102983510018056541193986125509967583449372944513686%
2014456854249783354063983457803507432659867546206587031391115109661576%
1516503950524064473134542712935711368808892154109606995841921790171765%
4893005967301395968656955403442490242166179811924853658409252017625986%
3906377380326741084120889097559846956027267560044221965159675319491944%
4329858601956832355668323043271718482856938415308981787209780110675524%
9261984676636588292027767560541974687923490737584058450361418191496554%
9380204996261913647129813735792291433712596853783222167675731678946910%
2544961918899665277861902312300283842359510665174220836052157261946625%
1978436609374754126564490582318478421494718185779380836210887341682756%
4604955635490336380176498300587838636742526281871869641022200525646024%
2806211868610934676397909273611681445706842617379544353781188146242432%
614618218707323073176507357593600}
\\

\hline \hline
\end{longtable}
  } \label{end20498s}

\nocite{MR2041227}
\nocite{MR0441665}
\nocite{MR1507868}
\nocite{MR1696313}
\nocite{MR0200258}
\nocite{MR0520733}
\nocite{MR2527719}
\nocite{MR2833514}
\nocite{MR3089699}
\nocite{MR0653502}
\nocite{MR2261030}
\nocite{MR937520}
\nocite{sage}
\nocite{OEIS}
\nocite{MR2833479}
\nocite{MR1329462}
\nocite{MR2476780}
\nocite{MR2483235}
\nocite{MR2888238}
\bibliography{sloane}
\bibliographystyle{plain}

Keywords: matrix, square, sequence, conjugacy class

\end{document}